\documentclass[10pt,twocolumn,twoside]{IEEEtran}
\usepackage{cite}
\usepackage{amsmath,amssymb,amsfonts}
\usepackage{graphicx}
\usepackage{algorithm,algorithmic}
\usepackage{hyperref}
\usepackage{bm}
\usepackage{colortbl}
\hypersetup{hidelinks=true}
\usepackage{textcomp}
\usepackage{booktabs}
\usepackage{makecell}

\usepackage{amsmath}
\usepackage{amssymb}

\usepackage{amsthm}
\usepackage{mathtools}
\usepackage{xcolor}
\usepackage{hyperref}
\hypersetup{
	colorlinks=true,
	linkcolor=blue,
	filecolor=magenta,      
	urlcolor=cyan,
	citecolor=blue
}

\usepackage{amsmath}

\DeclareMathOperator*{\argmin}{arg\,min}

\usepackage{tabularx}
\usepackage{pifont}
\usepackage{dsfont}
\def\cross{\ding{55}}

\newtheorem{theorem}{Theorem}[section]

\newtheorem{lemma}[theorem]{Lemma}

\newtheorem{assumption}[theorem]{Assumption}

\newtheorem{remark}[theorem]{Remark}

\def\Prm{\mathbf{X}}
\def\prm{\mathbf{x}}
\def\avgprm{\bar{\mathbf{x}}}
\def\dprm{\bm{\lambda}}

\def\bA{\mathbf{A}}

\def\wta{\bar{\bf A}}

\def\wtq{\bar{\bf Q}}
\def\wtk{{\bf K}}

\def\algoname{TiCoPD}

\def\data{\mathbb{P}}
\def\m{M}
\def\oneotimes{{\bf 1}_{\otimes}}
\def\oneotimesT{{\bf 1}_{\otimes}^{\top}}

\def\serr{{\bf e}_s}

\def\gerr{{\bf e}_g}

\def\wxv{{\bf W}_{xv}}

\def\rxi{{\bf R}}
\def\rxii{{\bf R}}

\def\a{{\tt a}}
\def\b{{\tt b}}
\def\c{{\tt c}}
\def\d{{\tt d}}
\def\e{{\tt e}}

\def\rrhomax{\rho_{\max}}
\def\rrhomin{\rho_{\min}}
\def\rhomax{\tilde{\rho}_{\max}}
\def\rhomin{\tilde{\rho}_{\min}}

\newcommand{\dotp}[2]{\left\langle{#1}\ \middle|\ {#2}\right\rangle}
\newcommand{\expec}[1]{\mathbb{E}\left[ {#1} \right]}

\newcommand*{\oscar}[1]{\textbf{\textcolor{magenta}{Oscar: #1}}}

\newcommand{\update}[1]{#1}

\allowdisplaybreaks
\begin{document}
	\title{Decentralized Stochastic Optimization over Unreliable Networks via Two-timescales Updates}
	\author{Haoming Liu, Chung-Yiu Yau, Hoi-To Wai, \IEEEmembership{Member, IEEE} 
    \textbf{\thanks{HL, CYY, HTW are with the Dept.~of SEEM, CUHK. This research is supported in part by project \#MMT-p5-23 of the Shun Hing Institute of Advanced Engineering, The Chinese University of Hong Kong.}}}
	
\maketitle
	
\begin{abstract}
This paper introduces a robust two-timescale compressed primal-dual ({\algoname}) algorithm tailored for decentralized optimization under bandwidth-limited and unreliable channels. By integrating the majorization-minimization approach with the primal-dual optimization framework, the {\algoname} algorithm strategically compresses the difference term shared among agents to enhance communication efficiency and robustness against noisy channels without compromising convergence stability. The method incorporates a mirror sequence for agent consensus on nonlinearly compressed terms updated on a fast timescale, together with a slow timescale primal-dual recursion for optimizing the objective function. Our analysis demonstrates that the proposed algorithm converges to a stationary solution when the objective function is smooth but possibly non-convex. Numerical experiments corroborate the conclusions of this paper.
\end{abstract}
	
	
\section{Introduction} \label{sec:introduction}
Many problems in signal processing and machine learning can be handled by tackling the optimization problem:
\begin{equation} \label{eq:origin_problem}
    \min_{\prm \in \mathbb{R}^{d}} ~f(\prm) := \frac{1}{n} \sum_{i=1}^n f_i(\prm) ,
\end{equation}
where for each $i \in [n] := \{1,\ldots,n\}$,
$f_i: \mathbb{R}^d \rightarrow \mathbb{R}$ is a continuously differentiable function, and $n$ is the number of agents. For example, $f_i$ is the cross-entropy loss for learning a neural network classifier model, defined on the $i$th partition of data. We consider a distributed optimization setting where $f_i$ is a local objective function held by the $i$th agent that is not shared with other agents. The local data available at the agents are typically heterogeneous, i.e., $f_i$ are distinct from others. Instead of relying on a central server that is connected to all agents, \eqref{eq:origin_problem} shall be handled by the cooperation between the $n$ agents. Distributed algorithms for \eqref{eq:origin_problem} have been developed actively, such as in signal processing for tackling the estimation problem in wireless sensor networks in privacy and bandwidth sensitive applications \cite{dimakis2010gossip, xiao2007distributed, kar2013consensus}, and in machine learning for parallel computation \cite{mateos2010distributed, lian2017can}. 

Decentralized optimization algorithms that can operate with on-device computation and local communication with a subset of agents are preferred for problem \eqref{eq:origin_problem}.
To this end, a popular approach is to mimick the (centralized) gradient method where agents perform gradient update with the local objective functions while performing a consensus step with neighbors. This leads to primal-only schemes such as decentralized gradient (DGD) method \cite{nedic2009distributed, shi2015extra, qu2017harnessing, zeng2018nonconvex, koloskova2020unified} as well as the popular extensions such as gradient tracking, EXTRA \cite{qu2017harnessing, shi2015extra, liu2024decentralized}.
\update{The other approach is to apply the alternating direction method of multipliers (ADMM) framework, e.g., \cite{zhang2014asynchronous, chang2016asynchronous} has developed ADMM algorithm on server-worker network, \cite{li2019communication, zhang2018admm, shi2014linear} applied ADMM on general decentralized network, and \cite{bastianello2020asynchronous} developed a relaxed ADMM algorithm for networks with potential link failures for strongly convex losses.}
The last approach is the (proximal) primal-dual algorithm \cite{hajinezhad2019perturbed, yi2021linear} which features a primal-dual descent-ascent method developed by viewing \eqref{eq:origin_problem} as a consensus-constrained optimization problem. The primal-dual framework also generalizes primal-only schemes such as EXTRA, grading tracking, and have been demonstrated to exhibit fast convergence under mild conditions \cite{chang2020distributed}.

For application scenarios such as wireless sensor networks, the communication channels between agents can be prone to \emph{(i) link failures}, \emph{(ii) bandwidth limitation}, and \emph{(iii) noise perturbation}. To handle optimization problems such as \eqref{eq:origin_problem}, an important goal is to design robust strategy for decentralized optimization algorithms.
The case of link failures, as modeled using time-varying networks, has been studied in \cite{nedic2014distributed, nedic2017achieving} for DGD and gradient tracking methods, \update{yet their results are limited as these algorithms require deterministic gradients.}
For bandwidth limitation, earlier works
\cite{wen2017terngrad, de2015taming, dean2012large} have combined distributed optimization algorithms with quantized communication that achieves acceptable performance in numerical experiments, \cite{seide20141, richtarik2021ef21} have studied error feedback subroutines with theoretically guaranteed exact convergence. More recently, \cite{reisizadeh2019exact, magnusson2020maintaining, liu2021linear, zhao2022beer} considered modifications of DGD and gradient tracking algorithms with error feedback subroutines, \cite{koloskova2019decentralized, koloskova2019decentralizedb, yau2022docom, xie2024communication} have studied their stochastic gradient extensions, and \cite{yau2024fully} developed a primal-dual algorithm that supports random with sparsified communication. 
For noisy perturbation in communication, \cite{srivastava2011distributed, reisizadeh2023dimix} considered a two-timescales update mechanism to modify the DGD algorithm for controlling the noise variance.
With a similar motivation, \cite{nassif2024differential, michelusi2022finite, nassif2023quantization} have studied the information theoretical limits under finite-bit transmission induced by noisy channels, \update{yet the latter require error-free binary communication channels}. 

\begin{table*}[t]
    \centering
    \resizebox{.98\linewidth}{!}{
    \begin{tabular}{l|c c c c c c}
    \toprule 
    Prior Works & \makecell{No Require.~on \\ Bdd Heterogeneity}  & \makecell{Time-varying or \\ Random Graph} & Compression & Noisy Comm. & \makecell{Convergence Rate \\ (Noiseless Comm.)} & \makecell{Convergence Rate \\ (Noisy Comm.)}\\
    \midrule
    DSGD \cite{koloskova2020unified} & \ding{55} & $\checkmark$ & \ding{55} & \ding{55} & $\mathcal{O}(\bar{\sigma}/\sqrt{nT})$ 
    & / \\[0.1cm]
    GT \cite{liu2024decentralized} & $\checkmark$ & \ding{55} & \ding{55} & \ding{55} & $\mathcal{O}(\bar{\sigma}/\sqrt{nT})$ & / \\[.1cm]
    CHOCO-SGD \cite{koloskova2019decentralized} & \ding{55} & \ding{55} & $\checkmark$ & \ding{55} & $\mathcal{O}(\bar{\sigma}/\sqrt{nT})$ &/ \\[0.1cm]
    DIMIX \cite{reisizadeh2023dimix} & \ding{55} & $\checkmark$  & $\checkmark$ & $\checkmark$ & / & $\mathcal{O}(T^{-1/3}\ln{T})$\\[0.1cm]
    CP-SGD \cite{xie2024communication} & $\checkmark$ & \ding{55} & $\checkmark$ & \ding{55} & $\mathcal{O}(\bar{\sigma}/\sqrt{nT})$ &/ \\[0.1cm]
    FSPDA-SA \cite{yau2024fully} & $\checkmark$ & $\checkmark$  & \ding{55}$^\dagger$ & \ding{55} & $\mathcal{O}(\bar{\sigma}/\sqrt{nT})$  &/ \\[0.1cm]
    \rowcolor{gray!10} {\bf This Work} & $\checkmark$ & $\checkmark$  & $\checkmark$ & $\checkmark$ & $\mathcal{O}(\bar{\sigma}/\sqrt{nT})$ &$\mathcal{O}(n^{-1/2}T^{-1/3})$ \\
    \bottomrule
    \end{tabular}}
    \vspace{.2cm}
    \caption{Comparison of {\algoname} to state-of-the-art algorithms with compressed communication in non-convex stochastic optimization. Convergence rate is measured w.r.t.~the smallest $\mathbb{E}[ \| \nabla f(\prm) \|^2 ]$ found in $T$ iterations. ($^\dagger$FSPDA supports noiseless random sparsification as the only means of communication compression.)}\vspace{-.4cm}
    \label{tab:compare}
\end{table*}

In addition to focusing on only one or two specific issues in unreliable networks, \update{there are also restrictions of the above decentralized algorithms with their theoretical guarantees that can lead to poor practical performance}. For instance, the convergence of CHOCO-SGD \cite{koloskova2019decentralized} requires a bounded heterogeneity assumption, \update{i.e., $\sup_{\prm \in \mathbb{R}^d} \| \nabla f_i( \prm) - \nabla f_j(\prm) \|$ is bounded}, and converges slowly when agent's data distribution are highly heterogeneous; CP-SGD \cite{xie2024communication} exhibits faster convergence, but does not support time-varying graphs and noisy compressed transmission; FSPDA \cite{yau2024fully} only supports random sparsification as compression scheme with noiseless communication.

Limitations of the existing works have motivated the current work to design fast-converging decentralized optimization algorithms that are simultaneously robust to unreliable networks.  
We notice a number of challenges. To combat data heterogeneity, one cannot rely on simple schemes such as the plain DGD. Instead, a possible solution is to utilize the primal-dual (PD) optimization framework that can naturally lead to a decentralized algorithm which converges regardless of the degree of data heterogeneity.
However, the PD algorithm design lacks a natural recipe to deal with link failures (time varying networks) and noisy communication as the algorithm is not integrated with an explicit consensus subroutine \update{that one can substitute with as in the design of CHOCO-SGD \cite{koloskova2019decentralized}}.

This paper designs a new decentralized algorithm integrating two ingredients \update{that are first applied to communication efficient optimization}: \emph{(A)} a majorization-minimization (MM) procedure and \emph{(B)} a two-timescale updating scheme for average consensus. 
The first component allows for randomness due to link failure in the communication network and separates the decision variables and the variables to be transmitted. 
\update{Our MM approach differs from that of \cite{di2016next} as we focus on improving the resiliency of decentralized algorithm over unreliable networks.}
The second component is inspired by the \emph{nonlinear gossiping} framework \cite{mathkar2016nonlinear} which replaces the classical linear average consensus map with a suitably designed nonlinear map for average consensus. The general idea is to observe that the nonlinear map converges at a faster rate than the stochastic approximation step in a consensus algorithm. 
In our case, we design the said nonlinear map with a compression operator applied to a \emph{compression error vector} to combat with limited bandwidth, while simultaneously allow for additive noise in the operation that handles noisy channels. 

This paper aims at developing a stochastic
primal-dual algorithm for \eqref{eq:origin_problem} that is robust to unreliable networks.
Our contributions are summarized as: 
\begin{itemize}
    \item We propose the {\bf T}wo-t{\bf i}mescale {\bf Co}mpressed stochastic {\bf P}rimal-{\bf D}ual ({\algoname}) algorithm. The {\algoname} algorithm follows a two-level update which separates \emph{communication} and \emph{optimization} steps, while handling them as a single-loop algorithm with two timescale updates. To our best knowledge, this is the first stochastic primal-dual algorithm for decentralized optimization on random graphs that supports general contractive compressors. 
    \item To incorporate compression into the decentralized algorithm, we develop a stochastic majorization-minimization (MM) procedure which suggests agents to transmit the compressed difference terms. This offers a new perspective for understanding the error feedback mechanism developed in \cite{koloskova2019decentralized} and draws connection to the nonlinear gossiping algorithm \cite{mathkar2016nonlinear}. We believe that this observation is of independent interest. 
    \item For optimization problems with continuously differentiable (possibly non-convex) objective functions, we show that with noiseless compressed transmission, the {\algoname} algorithm converges in expectation -- within $T$ iterations, it finds a solution whose squared gradient norm is ${\cal O}( 1 / \sqrt{nT} )$. The rate is comparable to that of a centralized SGD algorithm asymptotically. With noisy compressed transmission, we show that the {\algoname} algorithm finds a solution whose squared gradient norm is ${\cal O}( 1 / T^{1/3})$, within $T$ iterations. The rate is faster than \cite{reisizadeh2023dimix}, and does not require the bounded heterogeneity assumption.
\end{itemize} 
Table \ref{tab:compare} compares {\algoname} to existing algorithms. We remark that \update{existing works either impose restrictions on the random graph structure and compression scheme, or requiring strong assumptions on data heterogeneity, or they do not support noisy communication. These challenges are overcome simultaneously with the proposed {\algoname} algorithm.}
Compared to the conference version \cite{liu2025two}, we have extended the {\algoname} algorithm to work with stochastic gradients and unreliable networks, together with a new set of experiments. 

The rest of this paper is organized as follows. Section \ref{sec:Problem Statement} introduces the primal-dual framework of distributed optimization. Section \ref{sec:2TS Compressed Stochastic PDA} provides the derivation of two-timescales update and majorization-minimization, and develops the {\algoname} algorithm. Section \ref{sec:Convergence analysis} presents the convergence analysis of TiCoPD. Section \ref{sec:Numerical Experiments} provides the numerical experiments. Finally, Section \ref{sec:Conclusion} concludes the paper.

\emph{Notations.} 
Let $\wtk \in \mathbb{R}^{d \times d}$ be a symmetric matrix, the $\wtk$-weighted inner product of vectors ${\bf a,b} \in \mathbb{R}^d$ is denoted as $\dotp{{\bf a}}{{\bf b}}_{\wtk} := {\bf a}^\top \wtk {\bf b}$. \update{For positive (semi-)definite matrix $\wtk$, the $\wtk$-weighted (semi-)norm is denoted by $\| {\bf a} \|_{\wtk}^2 := {\bf a}^\top \wtk {\bf a}$.}
Lastly, $(\cdot)^\dagger$ is the Moore-Penrose inverse of the matrix,
$\lfloor \cdot \rfloor$ is the floor function,
$\otimes$ is the Kronecker product.

\section{Problem Statement} \label{sec:Problem Statement}
Let $G = (V,E)$ be an undirected and connected graph of the agent set $V = \{1,...,n\}$, and $E \subseteq V \times V$ is the set of usable edges between the $n$ agents. The graph $G$ is also endowed with an incidence matrix $\bA \in \mathbb{R}^{|E| \times n}$ and we shall denote $\wta := \bA \otimes {\bf I}_d$, where $\otimes$ denotes the Kronecker product.

We consider a setting where both the optimization problem \eqref{eq:origin_problem} and the graph are stochastic. Let the probability space be $(\Omega,\mathcal{F},\mathbb{P})$, where $\Omega := \Omega_1 \times \cdots \times \Omega_n \times \Omega_A$ is the sample space.
For each $i \in V$, $\xi_i \in \Omega_i$ admits the distribution $\mathbb{P}_i$ that leads to the $i$th local objective function:
\begin{equation}
	f_i( \prm_i ) := \mathbb{E}_{\xi_i \sim \data_i } [ f_i( \prm_i; \xi_i ) ].
\end{equation}
With a slight abuse of notation, we also used $f_i$ as the measurable function $f_i: \mathbb{R}^d \times \Omega_i \to \mathbb{R}$ such that its second argument takes a random sample.
To define the stochastic graph model, we consider $\xi_a \sim \mathbb{P}_A$ and define a subgraph selection diagonal matrix ${\bf I}( \xi_a ) \in \{0,1\}^{|E| \times |E|}$.
The corresponding subgraph is then denoted as ${\cal G}(\xi_a) = (V,E(\xi_a))$ with $E(\xi_a)\subseteq E$.
We have $\mathbb{E}_{ \xi_a \sim \mathbb{P}_A } [ {\bf I}( \xi_a ) ] > c {\bf I}_{|E|}$ for some $c>0$ such that each edge is selected with a non-zero probability. 
To simplify notation, we let $\Prm = [ \prm_1; \cdots; \prm_n ] \in \mathbb{R}^{nd}$.

Let $\wta(\xi_a) = ( {\bf I}(\xi_a) \bA ) \otimes {\bf I}_d$ be the incidence matrix for a \emph{random subgraph} of $G$, we observe that \eqref{eq:origin_problem} is equivalent to the stochastic equality constrained optimization problem:
\begin{equation} 
\min_{ \Prm \in \mathbb{R}^{nd}} ~ 
    \frac{1}{n} \sum_{i=1}^n f_i( \prm_i ) \quad \text{s.t.} \quad  \mathbb{E}_{\xi_a \sim \mathbb{P}_A } [\wta(\xi_a)] \Prm = {\bf 0} . \label{eq:main_problem}
\end{equation}
To facilitate our subsequent discussions, we denote $\nabla {\bf f}(\Prm;\xi ) \in \mathbb{R}^{nd}$ as the stack of the local stochastic gradients $\nabla f_i( \prm_i; \xi_i )$, and $\nabla {\bf f}(\Prm ) \in \mathbb{R}^{nd}$ as the stack of the local expected gradients $\nabla f_i(\prm_i)$. Moreover,
    \begin{align}
        &f( \avgprm ) := \frac{1}{n} \sum_{i=1}^n f_i( \avgprm ),~~\nabla f( \avgprm) := \frac{1}{n} \sum_{i=1}^n \nabla f_i( \avgprm ),
    \end{align}
 as the global objective function and exact global gradient evaluated on a common decision variable $\avgprm \in \mathbb{R}^d$. 


	

\subsection{Proximal Gradient Primal-dual Algorithm (Prox-GPDA)} \label{sec:pda}
\update{To introduce the principles of primal-dual algorithms and motivate {\algoname}, we consider extending Prox-GPDA \cite{hajinezhad2019perturbed} to the stochastic-constrained problem \eqref{eq:main_problem}}. 
We also highlight the difficulties in adapting Prox-GPDA and similar existing algorithms for unreliable networks.


\update{Our development begins by writing the augmented Lagrangian function of \eqref{eq:main_problem}. We focus on a sampled version of the augmented Lagrangian function}
\begin{equation} \label{eq:sample_lag}
{\cal L}( \Prm, \dprm; \xi ) := \frac{1}{n} \sum_{i=1}^n f_i( \prm_i ;\xi_i ) + \dprm^\top \wta(\xi_a) \Prm + \update{\theta} \|\wta(\xi_a) \Prm\|^2 ,
\end{equation} 
where $\dprm = ( \dprm_i )_{i \in E} \in \mathbb{R}^{ |E| d }$ is the Lagrange multiplier for the equality constraint and $\theta > 0$ is a regularization parameter. 
\update{Notice that the saddle points of the expected Lagrangian function $\mathbb{E}_{\xi} [ {\cal L}( \Prm, \dprm; \xi ) ]$ are the KKT points of \eqref{eq:main_problem} \cite{bertsekas1997nonlinear}.}

\update{The Prox-GPDA \cite{hajinezhad2019perturbed} can be derived from updating the primal-dual variables through finding the saddle points of $\mathbb{E}_{\xi} [ {\cal L}( \Prm, \dprm; \xi ) ]$. At iteration $t$,} we let $\xi^{t+1} \equiv ( \{\xi_i^{t+1}\}_{i=1}^n , \xi_a^{t+1}) \sim \mathbb{P}$ and $(\Prm^t, \dprm^t)$ be the $t$th primal-dual iterates. Linearizing $f_i(\cdot ; \xi_i^{t+1})$ and including the proximal term suggested in \cite{hajinezhad2019perturbed}, the primal-dual gradient descent-ascent steps on \eqref{eq:sample_lag} can be expressed as
\begin{align} 
\Prm^{t+1} &= \arg\min_{\Prm \in \mathbb{R}^{nd}}   \nabla {\bf f}(\Prm^t;\xi^{t+1})^\top (\Prm-\Prm^t)  +\Prm^\top \wta(\xi_a^{t+1})\dprm^t  \notag  \\
&  \hspace{-.2cm}  \quad + \update{\theta} \|\wta(\xi_a^{t+1}) \Prm \|^2 
+ \frac{1}{2} \| \Prm - \Prm^t \|_{ \widetilde{\alpha}^{-1} {\bf I} + \update{2 \theta} {\bf B}(\xi_a^{t+1}) }^2 , \label{eq:x_update_origin} \\
\dprm^{t+1} &= \arg\min_{\dprm \in \mathbb{R}^{nd} }  -\dprm^\top \wta(\xi_a^{t+1}) \Prm^t +\frac{1}{2\eta} \| \dprm - \dprm^t \|^2 , \label{eq:lambda_update_origin}
\end{align}
\update{where $\tilde{\alpha}>0, \eta>0$ are the step sizes.}
To simplify notations, we have also defined \update{the positive semidefinite matrix}
${\bf B} (\xi_a)= 2 \, \text{Diag}(\wta^\top \wta(\xi_a))-\wta^\top \wta(\xi_a)$.

\update{Observe that the sub-problems \eqref{eq:x_update_origin}, \eqref{eq:lambda_update_origin} are respectively $\tilde{\alpha}^{-1}$, $1/(2\eta)$ strongly convex,}
we obtain the following extended Prox-GPDA recursion: 
\begin{align} \label{eq:primal-dual update}
		\Prm^{t+1} &=  {\bf D}(\xi_a^{t+1})^{-1} \big( ({\bf I}_{nd} +\widetilde{\alpha} {\bf B}(\xi_a^{t+1}) ) \Prm^t \\
		& \qquad \qquad -\widetilde{\alpha}\nabla {\bf f}( \Prm^t;\xi^{t+1}) -\widetilde{\alpha}\wta(\xi_a^{t+1})^\top \dprm^t  \big) \notag \\
		\dprm^{t+1} &=  \dprm^t + \eta  \wta(\xi_a^{t+1}) \Prm^t,  \label{eq:primal-dual lambda update}
\end{align}
where ${\bf D}(\xi_a^{t+1}) = \update{4}\widetilde{\alpha}\theta \, \text{Diag}(\wta^\top \wta(\xi_a^{t+1}))+{\bf I}_{nd}$ is  diagonal. 
    
Remarkably, \eqref{eq:primal-dual update}, \eqref{eq:primal-dual lambda update} incorporate both \emph{communication} and \emph{optimization} steps in a single update. To see this, for any $i \in \{1, \ldots, n\}$, let $d_{ii}^{t+1}, L_{ii}^{t+1}$ be the $i$th diagonal element of ${\bf D}(\xi_a^{t+1}), \wta^\top \wta(\xi_a^{t+1})$, the local variable can be updated as
	\begin{equation} \notag
		\begin{split}
			& \prm_i^{t+1} = (d_{ii}^{t+1})^{-1} \underbrace{ [ (1 + \update{4} \widetilde{\alpha} \theta L_{ii}^{t+1} ) \prm_i^t - \widetilde{\alpha} \nabla f_i( \prm_i^t ; \xi_i^{t+1} ) ] }_{\text{\emph{(local) optimization step}}} \\
			& - (d_{ii}^{t+1})^{-1} \underbrace{ \left[ \update{2}\widetilde{\alpha} \theta \sum_{j \in {\cal N}_{i}^t} (\prm_j^t - \prm_i^t) + \widetilde{\alpha} \sum_{ (i,j) : j \in {\cal N}_i^t } \dprm_{ (i,j) }^t \right] }_{\text{\emph{communication step}}}
		\end{split}
	\end{equation}
and similar observations can be made for the $\dprm$-update. This modification of Prox-GPDA can be implemented on random subgraphs, although its convergence analysis differs significantly from that of \cite{hajinezhad2019perturbed}; see \cite{yau2024fully}. 

Most importantly, the extended Prox-GPDA \eqref{eq:primal-dual update}, \eqref{eq:primal-dual lambda update} is not suitable for unreliable networks \update{with potentially noisy or imperfect communication}, where it requires the agents to communicate a message of $d$ real numbers on each activated edge at every iteration. \update{
To resolve these issues, we develop the TiCoPD algorithm in the next section.
}
    
	
\section{Two-timescale Compressed Stochastic Primal-Dual (TiCoPD) Algorithm} \label{sec:2TS Compressed Stochastic PDA}
This section develops the {\bf T}wo-t{\bf i}mescale {\bf Co}mpressed stochastic {\bf P}rimal-{\bf D}ual ({\algoname}) algorithm for \update{handling the challenges of decentralized optimization in unreliable networks.}    
The algorithm depends on two ingredients: (i) a majorization-minimization (MM) step that introduces a \emph{surrogate} variable to separate the communication step from the optimization step \update{while handling potential link failures}, (ii) a \emph{two-timescales} update supporting \emph{noisy} and \emph{nonlinearly compressed} update of the surrogate variable \update{for handling the bandwidth limited and noisy communication issue}. 

We develop {\algoname} by revisiting the primal-dual update derived from \eqref{eq:sample_lag}. Observe that the gradient of $|| \wta(\xi_a) \Prm ||^2$ w.r.t.~$\prm_i$ is
\begin{equation} \label{eq:uncom_transmit} \textstyle 
\sum_{j \in {\cal N}_{i} (\xi_a) } ( \prm_i - \prm_j).
\end{equation}
Evaluating the above necessitates the neighbors' variables $\{ \prm_j \}_{j \in {\cal N}_i(\xi_a)}$. \update{This presents the major communication bottleneck in primal-dual algorithms such as Prox-GPDA.} 
\update{To remedy, our strategy is to sidestep $|| \wta(\xi_a) \Prm ||^2$ through relying on a surrogate sequence $\{ \hat\Prm^t \}_{t \geq 0}$ provided that
    $\hat{\Prm}^t$ approximates $\Prm^t$, and
    agent $i$ can acquire the neighbors' surrogate variables $( \hat{\prm}_j^t )_{j \in {\cal N}_i^t}$ efficiently, e.g., via compressed message exchanges.

In the following, we develop the {\algoname} algorithm in three steps through: (i) assuming that a surrogate sequence $\{ \hat\Prm^t \}_{t \geq 0}$ exists, we employ MM on the augmented Lagrangian function \eqref{eq:sample_lag} to derive the main PD algorithm while transferring the communication burden to that of $\hat{\Prm}^t$, (ii) develop an auxiliary recursion to construct the surrogate sequence $\{ \hat\Prm^t \}_{t \geq 0}$ over unreliable networks, (iii) finally, the two ingredients are combined to yield the {\algoname} algorithm that employs two-timescales updates to guarantee convergence.
}

\vspace{.2cm}
\noindent{\bf Step 1. Majorization-Minimization.}
\update{We focus on majorizing the $|| \wta(\xi_a) \Prm ||^2$ term in \eqref{eq:sample_lag} to yield a simpler update form in the primal-dual update. Fixing $\Prm^t$, there exists $M \geq 0$ (e.g., with $M = \| \wta \|_2$) such that for any $\Prm \in \mathbb{R}^{nd}$,}
\begin{equation} \notag
    \begin{aligned}
        || \wta(\xi_a) \Prm ||^2 & \leq ( \wta^\top \wta(\xi_a) \Prm^t )^\top ( \Prm - \Prm^t ) + \frac{M}{2} \| \Prm - \Prm^t \|^2.
    \end{aligned}
\end{equation}
\update{Introducing the surrogate variable $\hat{\Prm}^t$ and applying the Young's inequality yield}
\begin{equation} \notag
\begin{aligned}
& ( \wta^\top \wta(\xi_a) \Prm^t )^\top ( \Prm - \Prm^t ) = ( \wta^\top \wta(\xi_a) \hat{\Prm}^t )^\top ( \Prm - \Prm^t ) \\
& \qquad \qquad \qquad \qquad + ( \wta^\top \wta(\xi_a) ( \Prm^t - \hat{\Prm}^t ) )^\top ( \Prm - \Prm^t ) \\
& \leq ( \wta^\top \wta(\xi_a) \hat{\Prm}^t )^\top ( \Prm - \Prm^t ) \\
& \quad + \frac{1}{2} \| \wta^\top \wta(\xi_a) ( \Prm^t - \hat{\Prm}^t ) \|_F^2 + \frac{1}{2} \| \Prm - \Prm^t \|^2.
\end{aligned}
\end{equation}
\update{Applying the above majorization step leads to the following upper bound of \eqref{eq:sample_lag}:
\begin{equation}
\begin{aligned}
& \textstyle \frac{1}{n} \sum_{i=1}^n f_i( \prm_i ;\xi_i^{t+1} ) + \dprm^\top \wta(\xi_a^{t+1}) \Prm + {\rm constant} \\
& \textstyle + {\theta} \Big( \frac{M+1}{2} \| \Prm - \Prm^t \|^2 + ( \wta^\top \wta(\xi_a) \hat{\Prm}^t )^\top \Prm \Big),
\end{aligned}
\end{equation}
where we have omitted the constant terms. Further majorizing $f_i(\cdot)$ with its quadratic upper bound yields the $\Prm^{t+1}$-update:}
\begin{equation} \label{eq:x-update1}
\begin{aligned} 
		\Prm^{t+1} &= \argmin_{\Prm \in \mathbb{R}^{nd}} 
		\Prm^\top ( \nabla {\bf f}(\Prm^t;\xi^{t+1}) + \wta(\xi_a^{t+1})\dprm^t ) \\
		& \quad + \theta \Prm^\top \wta^\top \wta(\xi_a^{t+1}) \hat{\Prm}^t  + \frac{\theta \widetilde{\alpha} (\m+1) + 1 }{ 2 \widetilde{\alpha} } \| \Prm - \Prm^t \|^2 \\
        & \update{= \Prm^t - \frac{\widetilde{\alpha}}{\theta\widetilde{\alpha}(\m+1)+1} \big[ \nabla {\bf f}(\Prm^t;\xi^{t+1})  }\\
        & \qquad \qquad \update{+ \wta(\xi_a^{t+1})\dprm^t +\theta \wta^\top \wta (\xi_a^{t+1})\hat{\Prm}^t \big].}
\end{aligned}
\end{equation}
For the dual-update, we can similarly replace $\Prm^t$ with the surrogate variable $\hat{\Prm}^t$ to obtain:
\begin{equation}
\begin{aligned}
\dprm^{t+1} & = \argmin_{\dprm \in \mathbb{R}^{d|E|}} \frac{1}{2\eta} \| \dprm-\dprm^t \|^2 -\dprm^\top \wta(\xi_a^{t+1})\hat{\Prm}^t \\
& \update{= \dprm^t + \eta \wta(\xi_a^{t+1}) \hat{\Prm}^t. }
\label{eq:lambda-update1} 
\end{aligned}
\end{equation}

Finally, using the variable substitution $\widetilde{\dprm}^t = \wta^\top \dprm^t$, \update{we obtain the main algorithm} for the primal-dual updates:
\update{\begin{equation} \label{eq:pdc-recur}
    \begin{aligned}
        \prm^{t+1}_i =& \underbrace{\prm^t_i - \alpha \big[ \nabla f_i(\prm^t_i;\xi^{t+1}_i) + \widetilde{\dprm}^t_i\big]  }_{\text{\emph{(local) optimization step}}} \\
        &-\alpha\theta \underbrace{\sum_{j\in\mathcal{N}_i^t(\xi_a^{t+1})}(\hat{\prm}^t_j - \hat{\prm}^t_i) }_{\text{\emph{communication step}}}  \\
        \widetilde{\dprm}^{t+1}_i = &\widetilde{\dprm}^t_i + \eta \,  \underbrace{\sum_{j\in\mathcal{N}_i^t(\xi_a^{t+1})}(\hat{\prm}^t_j - \hat{\prm}^t_i)}_{\text{\emph{communication step}}} ,
    \end{aligned}
\end{equation}}with $\alpha = (\frac{1}{\widetilde{\alpha}}+\theta(\m+1) )^{-1}$. 
\update{Comparing with \eqref{eq:primal-dual update}, \eqref{eq:primal-dual lambda update},} it can be seen that the updates of $\Prm^t$ and $\widetilde{\dprm}^t$ require only sharing the surrogate variables $\hat{\Prm}^t$ \update{with neighboring agents, in lieu of the original variables $\Prm^t$. Thus if $\hat{\Prm}^t$ can be communicated effectively between agents, \eqref{eq:pdc-recur} would be a robust PD algorithm for decentralized optimization}. 

\vspace{.2cm}
\noindent{\bf Step 2. Surrogate Sequence $\{ \hat{\Prm}^t \}_{t \geq 0}$.}
Our next endeavor is to construct the sequence of surrogate variables $\{ \hat{\Prm}^t \}_{t \geq 0}$ that approximates $\{ \Prm^t \}_{t \geq 0}$ \update{with communication implemented over the unreliable networks}. To this end, we are inspired by the idea of nonlinear gossiping \cite{mathkar2016nonlinear} which splits an exact (consensus) operator into multiple small steps of nonlinear operations, while running in parallel with the main recursion; also see \cite{srivastava2011distributed,reisizadeh2023dimix} for related ideas.
In particular, we shall construct the surrogate variables using a recursion that tracks $\{ \Prm^t \}_{t \geq 0}$ and runs in parallel with \eqref{eq:pdc-recur}. Meanwhile, the design of this recursion should be bandwidth constrained such that it has to be implementable using compression. 

The bandwidth constraint of unreliable networks forbids us from setting $\hat{\Prm}^t = \Prm^t$ directly \update{as it requires $\{ \prm_i^t \}_{i=1}^n$ to be available at all agents}. Instead, we require the algorithm to only transmit \emph{compressed} data between agents. As a canonical example, we introduce the randomized quantization operator \cite{alistarh2017qsgd} which is represented as $\mathrm{qsgd}_s: \mathbb{R}^d \times \mathbb{R}^d \to \mathbb{R}^d$ such that:
\begin{equation} \label{eq:qsgd}
	\mathrm{qsgd}_s( \prm ; \xi_{q_i})=\frac{\|\prm\|}{s\tau}\cdot \mathrm{sign}(\prm) \odot \left\lfloor s\frac{|\prm|}{\|\prm\|}+\xi_{q_i} \right\rfloor\mathrm{~,}
\end{equation}
where $\odot$ denotes element-wise product, $s > 0$ is the number of precision levels, $\tau = 1 + \min\{ d/s^2, \sqrt{d}/s \}$ is a scaling factor and $\xi_{q_i} \sim {\cal U}[0,1]^d$ represents the dithering noise. 
Importantly, $\mathrm{qsgd}_s(\cdot)$ can be implemented with an \emph{encoder-decoder} architecture. On the transmitter's side, the \emph{encoder} compresses the $d$-dimensional input vector into a $d \log_2 s$-bits string alongside with the norm of vector $\| {\bm x} \|$. 
On the receiver's side, a \emph{decoder} converts the received bits into a quantized real vector in \eqref{eq:qsgd}. 
There are also a number of alternatives to the compressors like \eqref{eq:qsgd} -- including sparsifier that takes the top/random subset of coordinates from $\{1,...,d\}$ \cite{koloskova2019decentralized}, adaptive quantizer whose precision can be adjusted according to the input messages $\| {\bm x} \|$ \cite{michelusi2022finite}, etc.

\update{Unreliable networks may also introduce {noise} during the information transmission step between agents. For example, during the retrieval of $\mathrm{qsgd}_s(\cdot)$ at the decoder, the received bits may be errorneous and the full-precision scaling factor $\| {\bf x} \|$ may be distorted. 
In the latter case, suppose that $\prm_i$ is the intended message to be transmitted, we model the received message as $\mathrm{qsgd}_s ( \prm_i ;\xi_{q_i}) + {\bm w}_i$ such that ${\bm w}_i$ is a zero-mean {additive noise}. Observe that the intended message has been distorted twice --- $\mathrm{qsgd}_s(\cdot)$ introduced \emph{compression noise}, ${\bm w}_i$ represents the \emph{communication noise}. As an important feature, the \emph{communication noise} may not vanish when $\| \Prm \| \to 0$, while the \emph{compression noise} will.

To model the general \emph{noisy and contractive} compressed communication process, we consider a compression operator ${Q}:\mathbb{R}^{nd} \times \Omega_q^n \to \mathbb{R}^{nd}$ satisfying:}
\begin{assumption} \label{assm:compress}
For any fixed $\Prm \in \mathbb{R}^{nd}$, there exists $\delta \in (0,1]$\update{, $r> 0$} and $\sigma_{\xi} \geq 0$ such that 
\begin{equation} \label{eq:noisy_comp}
Q(\Prm; \xi_q) = \update{\widehat{Q}}( \Prm ; \bar\xi_q ) + {\bm W},
\end{equation}
where  $\mathbb{E}[ {\bm W} ] = {\bm 0}$ and $\mathbb{E}[ \|{\bm W}\|^2 ] \leq \sigma_{ \xi }^2$. Moreover, 
\update{\begin{equation}
\begin{aligned}
& \mathbb{E}_{ \bar{\xi}_q } \| \Prm -  {\widehat{Q} (\Prm; \bar\xi_q)} / {r} \|^2  \le (1-\delta)^2  \| \Prm \|^2 .
\end{aligned}
\label{eq:Q-contraction}
\end{equation}}
\end{assumption}
\noindent\update{In \eqref{eq:noisy_comp}, the \emph{compression noise} and \emph{communication noise} are modeled in $\widehat{Q}(\cdot)$ and ${\bm W}$, respectively. Particularly, when $\widehat{Q}(\cdot) \equiv {\rm qsgd}_s(\cdot)$, \eqref{eq:Q-contraction} holds with $\delta = \frac{1}{2 \tau}$, $r=1$, $\tau = 1 + \min\{ d/s^2, \sqrt{d}/s \}$ \cite{koloskova2019decentralized}. 
In addition, if $\widehat{Q}(\cdot)$ is the top-$k$ sparsifier, then \eqref{eq:Q-contraction} holds with $\delta = k/d, r=1$, see \cite{koloskova2019decentralized} for other examples. For adaptive quantization, we also note that a similar condition holds in \cite[Definition 3]{michelusi2022finite}. The parameter $r>0$ enables the use of dynamic scaling compressors \cite{liao2024robust}. 
We notice that the above condition is more general than existing works. For example, \cite{reisizadeh2023dimix} require ${Q}(\cdot)$ to be unbiased and the overall error (i.e., compression and communication) to be bounded for any $\Prm$.}


Equipped with $Q(\cdot)$ in \eqref{eq:noisy_comp}, we propose to construct \update{$\hat{\Prm}^t$ which approximates $\Prm^t$} iteratively by 
\begin{equation} \label{eq:xhat_update}
	\hat{\Prm}^{t} = \hat{\Prm}^{t-1} + \gamma Q(\Prm^{t} - \hat{\Prm}^{t-1} ;\xi^{t}_q),
\end{equation}
where $\gamma>0$ is a stepsize parameter controlling the transmission noise in $\hat{\Prm}^t$.
Notice that when $\Prm^t = \Prm$ is fixed, 
\update{by Assumption~\ref{assm:compress}, it can be shown that the recursion \eqref{eq:xhat_update} leads to $\lim_{t \to \infty} \mathbb{E} [ \| \hat{\Prm}^t - \Prm \|^2 ] \leq \gamma \sigma_{\xi}^2 / \delta $. 
For example, see the proof of Theorem 2 in \cite{koloskova2019decentralized}.}


\vspace{.2cm}
\noindent {\bf Step 3. {\algoname} Algorithm.} 
\update{For the primal-dual update \eqref{eq:pdc-recur} to converge, we need $\hat{\Prm}^t$ to closely track $\Prm^t$. As discussed after \eqref{eq:xhat_update}, the tracking condition can be achieved if $\Prm^t$ and $\dprm^t$ vary slowly relatively to the convergence of $\hat{\Prm}^t$ towards $\Prm^t$, which can be controlled respectively by the step size parameters $(\alpha, \eta)$ and $\gamma$. Overall, this suggests a \emph{two-timescales update} \cite{borkar1997stochastic} that designates the main algorithm \eqref{eq:pdc-recur} as the `slow' upper level (UL) updates and  the auxiliary recursion \eqref{eq:xhat_update} as the `fast' lower level (LL) updates.}
We propose the {\algoname} algorithm as:
\begin{align}
 & \begin{cases} 
    \Prm^{t+1} =& \hspace{-.2cm} \Prm^t - \alpha \big( \nabla {\bf f}(\Prm^t;\xi^{t+1}) +\widetilde{\dprm}^{t} +\theta \wta^\top \wta(\xi_a^{t+1})\hat{\Prm}^t \big) ,  \\
	\widetilde{\dprm}^{t+1} = & \hspace{-.2cm}\widetilde{\dprm}^t + \eta \wta^\top \wta(\xi_a^{t+1}) \hat{\Prm}^t,
    \end{cases} \notag \\
& \begin{cases}
    \hat{\Prm}^{t+1} = \hat{\Prm}^{t} + \gamma Q(\Prm^{t+1} - \hat{\Prm}^{t} ;\xi_q^{t+1}).
    \end{cases} \label{eq:pdc-recur-all}
\end{align}  
\update{We anticipate the algorithm to converge as a \emph{two-timescales} stochastic approximation algorithm by selecting the step sizes as $\gamma \gg \max\{ \alpha, \eta \}$.}
Note that {\algoname} is also a special case of the nonlinear gossip algorithm \cite{mathkar2016nonlinear}. 

We observe from \eqref{eq:pdc-recur-all} that in the UL updates, the $i$-th local variable $\prm_i^t, \widetilde{\dprm}_i^t$ requires aggregating the surrogate variables from \emph{active neighbors} $\sum_{j \in {\cal N}_i^t} ( \hat{\prm}_j^t - \hat{\prm}_i^t )$. Meanwhile, to incorporate communication compression, the LL updates for the surrogate variable require the \emph{(noisy) compressed difference} $Q(\Prm^{t+1} - \hat{\Prm}^{t} ;\xi_{q}^{t+1})$ to be transmitted. Accounting for the fact that the neighbor set may vary at each iteration in the time-varying  graph setting, this can be achieved by storing a local copy for $\hat{\prm}_j^t$, $\hat{\prm}_i^t$ at each neighbor $j \in {\cal N}_i$ at agent $i$, denoted respectively by $\hat{\prm}_{i,j}^t, \hat{\prm}_{j,i}^t$. We summarize the implementation details in the psuedo code Algorithm \ref{alg:TiCoPD}.

\begin{algorithm}[t] 
\caption{{\algoname} Algorithm} 
\begin{algorithmic}[1]\label{alg:TiCoPD}
    \STATE {\bf Input:} Algorithm parameters $\alpha,\theta,\eta,\gamma$, 
    initialization $\Prm^0,\widetilde{\lambda}^0$,$\hat{\prm}^0_{i,j}=\hat{\prm}^0_{j},i,j\in[n]$.
    \FOR{$t = 0,\ldots,T-1$}
    \STATE Draw a sample for time-varying graph $\xi_a^{t+1} \sim \data_a$.
    \FOR{each agent $i\in [n]$}
        \STATE Draw the samples for stochastic gradient $\xi_i^{t+1} \sim \data_i$ and random (noisy) compressor $\xi_{q,i}^{t+1} \sim \data_q$.
        \STATE Receive the compressed differences from $j \in {\cal N}_i^t$ of $Q( \prm_j^t - \hat{\prm}_{i,j}^{t-1} ; \xi_{q,j}^{t} )$ and update
        \[
        \hat{\prm}_{i,j}^{t} = \hat{\prm}_{i,j}^{t-1} + \gamma Q( \prm_j^t - \hat{\prm}_{i,j}^{t-1} ; \xi_{q,j}^{t} ), j \in {\cal N}_i^{t}
        \]
        Note that the above step is also performed at the \emph{transmitter} side of $j \in {\cal N}_i^t$.
        \STATE Perform the primal-dual update:
            \begin{align}
            \prm^{t+1}_i &= \prm^t_i - \alpha \big(\nabla f_i(\prm_i^t ;\xi_i^{t+1}) \notag \\
            &\quad  \textstyle + \widetilde{\lambda}_i^{t}+\theta \sum_{j\in \mathcal{N}_i^{t}} (\hat{\prm}^t_{j,i} -\hat{\prm}^t _{i,j} )\big) ,\\
            \widetilde{\lambda}^{t+1}_i &= \textstyle \widetilde{\lambda}^t_i + \eta \sum_{j\in \mathcal{N}_i^{t}} (\hat{\prm}^t_{j,i} - \hat{\prm}^t _{i,j} ).
            \end{align}
        \ENDFOR
    \ENDFOR
\end{algorithmic}
\end{algorithm}

\begin{remark} We remark about an interesting conceptual differences between {\algoname} and CHOCO-SGD \cite{koloskova2019decentralized}. The latter implements a compressed DSGD algorithm through two steps: adaptation and consensus, i.e.,
\begin{align}
& \hat{\Prm}^{t+1}=\hat{\Prm}^{t} + Q ({ {\Prm^{t}-\eta_t \nabla {\bf f} (\Prm^{t}; \xi^{t+1} )}}-\hat{\Prm}^{t}; \xi_q^{t+1} ) , \notag \\
& \Prm^{t+1}={ {\Prm^{t}-\eta_t \nabla {\bf f} (\Prm^{t}; \xi^{t+1} )}}+\gamma ( {\bf W} -{\bf I}) \hat{\Prm}^{t+1}. \notag
\end{align}
where ${\bf W}$ is doubly stochastic and serves as a weighted adjacency matrix for the static graph $G$.
We observe that CHOCO-SGD is a dynamical consensus modification of CHOCO-GOSSIP in \cite{koloskova2019decentralized}. Particularly, it produces a consensual vector that  averages $\Prm^{t}-\eta_t \nabla {\bf f} (\Prm^{t}; \xi^{t+1} )$.
The design can be easily adapted to algorithms that depend on consensual variables such as gradient tracking, see the recent works \cite{zhao2022beer, yau2022docom}. 
Meanwhile, it is challenging to apply the CHOCO-SGD's recipe on primal-dual algorithms such as \eqref{eq:primal-dual update}, \eqref{eq:primal-dual lambda update} where there is no clear `averaging step' in the algorithm. 

The {\algoname} algorithm departed from the CHOCO-SGD's recipe and was derived through a majorization-minimization procedure. 
As a result, {\algoname} takes the variable $\hat{\Prm}$ as a surrogate for $\Prm$ and the involved recursion is a fixed point iteration to achieve $\hat{\Prm} = \Prm$.
To our best knowledge, {\algoname} is the first algorithm to consider such two timescales update in compressed decentralized optimization.
\end{remark}
	
\section{Convergence analysis}  \label{sec:Convergence analysis}
This section establishes the convergence of the {\algoname} algorithm towards a stationary point of \eqref{eq:main_problem} at a sublinear rate. As a preparation step, we first define the average decision variable as
\begin{equation} \label{eq:avgprm} \textstyle
\avgprm^t = \frac{1}{n} \sum_{i=1}^n \prm_i^t = \frac{1}{n} ( {\bf 1}^\top \otimes {\bf I}_d ) \Prm^t,
\end{equation}
and using the consensus error operator $\wtk := ( {\bf I}_n - {\bf 1}{\bf 1}^\top/n ) \otimes {\bf I}_d$, the consensus error can be expressed as
\begin{equation} \textstyle
	\sum_{i=1}^n \| \prm_i^t - \avgprm^t \|^2 = \| \wtk \Prm^t \|^2 = \| \Prm^t \|_{\wtk}^2,
\end{equation}
and $\wtq := ( \wta^\top \rxii \wta )^\dagger$, where $(\cdot)^\dagger$ denotes the Moore-Penrose inverse \update{and $\rxii := \expec{{\bf I}(\xi_a)}$ is the diagonal matrix of edge activation probabilities}. We observe that $\wta^\top \rxii \wta \wtk = \wta^\top \rxii \wta = \wtk \wta^\top \rxii \wta$ and $\wtq \wta^\top \rxii \wta = \wta^\top \rxii \wta \wtq = \wtk$. 

We further make the following assumptions about the objective functions and their stochastic gradients. 
\begin{assumption} \label{assm:lip}
For any $i \in [n]$, the function $f_i$ is $L$-smooth such that
\begin{equation} \label{eq:f_lip}
\|\nabla f_i({\bf x}) - \nabla f_i({\bf y}) \| \leq L \|{\bf x} - {\bf y} \|, ~\forall~ {\bf x},{\bf y} \in \mathbb{R}^{d}.
\end{equation}
There exists $f^\star > -\infty$ such that $f_i( {\bf x} ) \geq f^\star$ for any ${\bf x} \in \mathbb{R}^d$.
\end{assumption}
	\begin{assumption} \label{assm:f_var} For any $i \in [n]$ and fixed ${\bf y} \in \mathbb{R}^d$, there exists $\sigma_i \geq 0$ such that
		\begin{equation} \label{eq:f_var}
			\mathbb{E}[\| \nabla f_i({\bf y}; \xi_i ) - \nabla f_i({\bf y}) \|^2] \leq \sigma_i^2.
		\end{equation}
    Moreover, $\{ \xi_i \}_{i=1}^n$ are mutually independent. To simplify notations, we define $\bar{\sigma}^2=1/n \sum_{i=1}^n \sigma_i^2$.
	\end{assumption}
    \noindent \update{Assumption \ref{assm:lip} is a standard smoothness condition for the continuously differentiable objective function, and Assumption \ref{assm:f_var} requires that stochastic gradient has bounded variance.}
	
	\begin{assumption} \label{assm:rand-graph}
		There exists constants $\rrhomax \ge \rrhomin> 0$ and $\rhomax \ge \rhomin >0$ such that 
		\begin{equation} \label{eq:q_ineq}
			\begin{aligned} 
				& \rrhomin \wtk \preceq \wta^\top \rxii \wta  \preceq \rrhomax \wtk , \\
                & \rhomin \wtk \preceq \wta^\top \wta \preceq  \rhomax \wtk .
			\end{aligned}
		\end{equation}
	\end{assumption}
	\noindent Notice that Assumption~\ref{assm:rand-graph} holds if ${\rm diag}(\rxii) > {\bf 0}$, i.e. when each edge is selected with a positive probability.
    Furthermore, as a consequence, it holds $\rrhomax^{-1} \wtk \preceq \wtq \preceq \rrhomin^{-1} \wtk$. 
	
	\begin{assumption} \label{assm:graph_var}
		For any fixed $\Prm \in \mathbb{R}^{nd}$, 
		\begin{equation}
			\mathbb{E} \left[ \|\wta^\top \wta(\xi_a) \Prm - \wta^\top \rxii \wta \Prm \|^2 \right] \leq \sigma_A^2 \|\wta \Prm \|^2_{\bf R} . \label{eq:graph_var}
		\end{equation}
	\end{assumption}
\noindent It bounds the variance of $\wta^\top \wta(\xi_a) \Prm$. \update{The assumption can be guaranteed as long as the probability of selecting an edge in the random graph is strictly positive for every edge, including the extreme case when only one edge is active in the random graph. In particular, it can be satisfied with the variance parameter $\sigma_A^2 = \expec{ \| \wta^\top (\xi_a) \rxi^{-1} - \wta^\top \|^2 \| \rxi \|}$.}
	

The above assumptions guarantee the convergence of {\algoname} algorithm towards a stationary point of \eqref{eq:main_problem}:
\begin{theorem} \label{theo:main_theorem}
    Under Assumptions \ref{assm:compress}, \ref{assm:lip}--\ref{assm:graph_var}, we set the step sizes and parameters as $\theta \ge \theta_{lb}, \alpha \le \alpha_{ub}, \gamma \le \update{\frac{1}{r}}$ where
    \begin{equation} \label{eq:stepsize-thm}
        \begin{aligned}
            &\eta = \frac{\update{r}\gamma\delta}{8\rhomax}, ~ \theta_{lb} = \frac{4}{\rrhomin} \max \left\{\frac{2L^2}{ n\a},\frac{2048\rhomax}{\update{r}\gamma\delta\rrhomin},L^2 \right\} ,\\
            &\alpha_{ub} = \frac{\update{r}\gamma\delta}{256\theta} \cdot \min \Bigg\{\frac{\rrhomin^2}{\rrhomax^2\rhomax},\frac{\rrhomin^2}{\sigma_{A}^2\rrhomax\rhomax},  \\
            &\qquad \qquad \qquad \qquad \frac{1}{72n\a \rhomax},\frac{\rrhomin}{2\rhomax \a} \Bigg\} .
        \end{aligned}
    \end{equation}
    Then, for any $T \geq 1$, it holds  
    \begin{align}
        & \frac{1}{T} \sum_{t=0}^{T-1}\expec{\left\| \nabla f(\avgprm^{t}) \right\|^2} \leq \frac{F_0 - f^\star}{\alpha T/16} \notag + 16\alpha\omega_{\sigma}\bar{\sigma}^2 \\
        &\qquad \qquad \qquad \qquad \qquad+ \frac{ 128\a \gamma^2\sigma_{\xi}^2\update{\rhomax} }{ \alpha \update{\rrhomin} },\\
        &\frac{1}{T} \sum_{t=0}^{T-1}\expec{ \| \Prm^t \|_{\wtk}^2} \le  \frac{4(F_0 - f^\star)}{\alpha\theta\rrhomin \a T} + \frac{4\alpha\omega_{\sigma}}{\theta\rrhomin \a} \bar{\sigma}^2 +\frac{32 \gamma^2\sigma_{\xi}^2}{\alpha\theta\rrhomin },
    \end{align}
    where
    $\omega_{\sigma}=\frac{L}{2n} + \a{\cal O}(\frac{n}{\update{r}\gamma\delta})$, 
    $F_0 = f( \avgprm^0 ) + \a {\cal O} ( \| \Prm^0 \|_{\wtk}^2 +  \frac{\alpha}{\eta}\| \widetilde{\lambda}^0 \|_{\wtk}^2 + {\textstyle \frac{\alpha}{\eta} } \|   \nabla {\bf f} ( \avgprm^0 ) \|_{ \wtk }^2 )$. 
    \update{The above statements hold for any $\a >0$, which is a free quantity that may be related to $T$.}
\end{theorem}
\noindent 
Note that the constraint on the compressed parameter update stepsize $\gamma$ is implicit. 
\update{The proof outline of Theorem \ref{theo:main_theorem} will be provided in Sec.~\ref{sec:pf_outline}, while the supporting lemmas will be proven in the appendix.}


\begin{remark}
We recall from the derivation in Section~\ref{sec:2TS Compressed Stochastic PDA} that $\alpha$ is implicitly related to $\widetilde{\alpha}$. To see that {\algoname} is still a stochastic primal-dual algorithm for the augmented Lagrangian function with $\widetilde{\alpha} > 0$ under the step size choices in Theorem~\ref{theo:main_theorem}. We note that upon fixing $\theta$, there exists a stepsize $\alpha > 0$ satisfying the parameter choices in the theorem and $\alpha=\frac{1}{\frac{1}{\widetilde{\alpha}}+\theta(\m+1)}$, provided that $\widetilde{\alpha}>0$ is sufficiently small. 
\end{remark}

In the following, we discuss the consequences of Theorem \ref{theo:main_theorem} through specializing it to various cases and derive the specialized convergence rates of {\algoname}. We shall highlight how choosing the stepsizes $\alpha, \gamma$ on different timescales under \eqref{eq:stepsize-thm} lead to a convergent decentralized algorithm. Particularly, {\algoname} achieves state-of-the-art convergence rates in all settings with unreliable networks. The results are summarized in Table~\ref{tab:summary}.
Hereby, we concentrate on the setting with a fixed iteration number $T \gg 1$, and define the random variable \update{${\sf R}$} as uniformly and independently drawn from $\{ 0, 1, \ldots, T-1 \}$. 
\update{
Following \cite{ghadimi2013stochastic}, ${\sf R}$ can be seen as a random stopping criterion for {\algoname}. Below, we shall demonstrate that $\avgprm^{\sf R}$ is an efficient estimator of stationary solutions to \eqref{eq:origin_problem}.
}

\paragraph{Noiseless Communication ($\sigma_{\xi} = 0$)}
Consider the case when communication channel is noiseless but the communicated messages can be compressed ($\delta>0$). This is the most common setting considered in the literature, e.g., \cite{koloskova2019decentralizedb}.
Particularly, as the standalone $\hat{\Prm}$-update \eqref{eq:xhat_update} admits linear convergence with $\gamma = 1$, we anticipate that {\algoname} to take a constant $\gamma$ for optimal performance. 

When $\bar{\sigma} > 0$ such that the algorithm takes \emph{noisy gradients}, by selecting $\alpha= {\cal O}(\sqrt{n/ (\bar{\sigma}^2 T) }), \theta = {\cal O}(\sqrt{T}), \gamma = 1, \a = {\cal O}(1/{\sqrt{T}})$, it follows
\begin{align} \label{eq:converge-Rstep}
    \expec{\left\| \nabla f(\avgprm^{\update{{\sf R}}}) \right\|^2} = {\cal O} \big( \sqrt{ \bar{\sigma}^2 / (nT) } \big), 
\end{align}
and  $
    \expec{ \| \Prm^\update{{\sf R}} \|_{\wtk}^2} = {\cal O} \big( 1 / T \big)$ with $\a = 1$.
\update{Given any $\epsilon > 0$, we can deduce that by setting $T = \Omega( 1/ \epsilon^2 )$ and the appropriate $\alpha ,\theta ,\a ,\gamma$ that might be associated with $T$, the TiCoPD algorithm is guaranteed to generate an $\epsilon$-stationary solution in $T$ iterations (in expectation).}
Observe that with a diminishing $\alpha$ and constant $\gamma$, {\algoname} achieves the so-called linear speedup such that its convergence rates are asymptotically equivalent to that of centralized SGD with a minibatch size of $n$ on \eqref{eq:main_problem}. Moreover, they are comparable to that of decentralized algorithms such as DGD \cite{zeng2018nonconvex} and CHOCO-SGD \cite{koloskova2019decentralizedb} \emph{without} requiring additional conditions such as bounded gradient heterogeneity. 

On the other hand, when $\bar{\sigma} = 0$ such that the algorithm takes \emph{exact gradients}, it is possible to adopt a constant $\alpha$ as well. We have \update{$\theta_{lb} \asymp \delta^{-1}$}, $\alpha_{ub} \asymp \delta^2$, and thus
\begin{align} \notag
    \expec{\left\| \nabla f(\avgprm^\update{{\sf R}}) \right\|^2} = {\cal O} \left( {1} / ({ \delta^2 T}) \right), 
    \expec{ \| \Prm^\update{{\sf R}} \|_{\wtk}^2} = {\cal O} \big( {1} / ({ \delta T}) \big).
\end{align}
Recall that $\delta \in (0,1]$ of Assumption~\ref{assm:compress} is affected by the quality of the compressor.
For example,
the upper bound on $\expec{\left\| \nabla f(\avgprm^\update{{\sf R}}) \right\|^2}$ evaluates to ${\cal O} ( d s^{-2} T^{-1} )$ for the case of randomized quantization.

\paragraph{Noisy Communication ($\sigma_{\xi} > 0$)} In this case, we expect the convergence rate to be slower as the algorithm needs to control the communication noise $\sigma_{\xi}$ by adopting a decreasing stepsize for $\gamma$ even in the standalone update \eqref{eq:xhat_update}. This noisy communication setting has only been considered in a small number of works, e.g., \cite{srivastava2011distributed, reisizadeh2023dimix, nassif2024differential}. 

With \emph{noisy gradients} ($\bar{\sigma}>0$), by choosing the step sizes and parameters as $\alpha={\cal O}(
T^{-\frac{2}{3}}), \theta ={\cal O}(T^{\frac{1}{3}}),\a={\cal O}(T^{-\frac{1}{3}}),\gamma={\cal O}(T^{-\frac{1}{3}})$, it holds
\begin{align} \notag
    \expec{\left\| \nabla f(\avgprm^\update{{\sf R}}) \right\|^2} = 
    {\cal O} \left( \frac{1 + \sigma_{\xi}^2}{ T^{\frac{1}{3}}} \right),
    \expec{ \| \Prm^\update{{\sf R}} \|_{\wtk}^2} = {\cal O} \left( \frac{1}{T^{\frac{1}{3}}} \right).
\end{align}
We notice that the convergence rates are only at ${\cal O}( 1/T^{1/3} )$ which are slower than the previous case of noiseless communication. 
That said, when compared to DIMIX \cite{reisizadeh2023dimix} which achieves the rate of $ {\cal O} \big( T^{-\frac{1}{3}+\epsilon} \big)$, $\epsilon > 0$ for the averaged squared gradient norm, our convergence rate is slightly faster. In addition, the convergence of the DIMIX algorithm requires strong assumptions such as bounded heterogeneity and \emph{a-priori} bounded iterates when used with random-$k$ sparsification and random quantization compressors. The latter restrictions are not found in our results for {\algoname}. 

Lastly, we study if \emph{exact gradient} (i.e., $\bar{\sigma} = 0$) may lead to faster convergence. By choosing $\alpha={\cal O}( T^{-\frac{2}{3}} ), \theta ={\cal O}(T^{\frac{1}{3}}),\a={\cal O}(T^{-\frac{1}{3}}),\gamma={\cal O}(T^{-\frac{1}{3}})$, one only has
\begin{align} \notag
    \expec{\left\| \nabla f(\avgprm^\update{{\sf R}}) \right\|^2} = {\cal O} 
    \left( \frac{1 + \sigma_{\xi}^2}{T^{\frac{1}{3}}}\right),
    \expec{ \| \Prm^\update{{\sf R}} \|_{\wtk}^2} = {\cal O} \left( \frac{1}{T^{\frac{1}{3}}} \right),
\end{align}
i.e., similar to the case with noisy gradient. 



\begin{table}[t]   
\begin{center}   
\resizebox{1\linewidth}{!}{
\begin{tabular}{c c l l}   
\toprule
\multicolumn{2}{c}{\bfseries Noise present?} & \multicolumn{2}{c}{\bfseries Convergence Rates of {\algoname}} \\
\bfseries Grad.~($\bar{\sigma}$) & \bfseries Comm.~($\sigma_{\xi}$) & \bfseries Grad.~$\mathbb{E}[ \| \nabla f( \avgprm^\update{{\sf R}} ) \|^2 ]$ & \bfseries Cons.~$\mathbb{E}[ \| \Prm^\update{{\sf R}} \|_{\wtk}^2 ]$ \\   
\midrule 
\cross & \cross & ${\cal O} \left( {1} / ({ \delta^2 T}) \right)$ & ${\cal O} ( {1} / ({\delta T}) )$  \\ 
\hline   
\checkmark & \cross & ${\cal O} \big( {\bar{\sigma}} / ( {n ^{\frac{1}{2}} T^{\frac{1}{2}}} ) \big)$ & ${\cal O} \big( {1} / {T} \big)$ \\
\hline  
\cross & \checkmark & ${\cal O} \big( ({1 + \sigma_{\xi}^2}) / { T^{\frac{1}{3}}}  \big)$ & ${\cal O} \big( 1 / T^{\frac{1}{3}} \big)$  \\ 
\hline  
\checkmark & \checkmark & ${\cal O} \big( { (1 + \sigma_{\xi}^2 ) } / {T^{\frac{1}{3}}}  \big)$ &  ${\cal O} \big( 1 / T^{\frac{1}{3}} \big)$ \\ 
\bottomrule
\end{tabular}}   
\vspace{.2em}
\caption{Convergence rates of {\algoname} on unreliable networks under different scenarios.}\vspace{-.4cm} \label{tab:summary}
\end{center}   
\end{table}

\subsection{Proof outline of Theorem \ref{theo:main_theorem}}\label{sec:pf_outline}


Our plan is to analyze     
the stable point of TiCoPD through studying the progress of $f(\avgprm^{t})$ and to control the gradient error by providing an upper bound of $\| \nabla f (\avgprm^t) \|^2$. 
Unlike previous works such as \cite{yau2024fully}, our analysis needs to deal with the surrogate variable $\hat{\Prm}$ and handle the noise effects in compressed communication.
For notational convenience, we denote $\mathbf{v}^t = \alpha \widetilde{\lambda}^{t} + \alpha \nabla {\bf f}(({\bf 1}_n \otimes {\bf I}_d ) \avgprm^{t})$ as a measure for the tracking performance of the average exact gradient.
    
The first step is to establish the following descent lemma:
\begin{lemma} \label{lemma:descent}
Under Assumption \ref{assm:lip} and \ref{assm:f_var}, when $\alpha \leq \frac{1}{4L}$,
{\begin{align} 
	&\mathbb{E}\left[ f (\avgprm^{t+1}) \right] \leq  
	\expec{ f (\avgprm^{t})} - \frac{\alpha}{4} \expec{\left\| \nabla f (\avgprm^t) \right\|^2}  \notag \\
	& \quad + \frac{\alpha L^2}{n} \expec{\| \Prm^t \|^2_{\wtk}} + \frac{\alpha^2 L}{2n^2}\sum_{i=1}^n \sigma_i^2 . 
	\label{eq:descent_lemma}
\end{align}}
\end{lemma}
\noindent  See Appendix \ref{app:lemma_descent_proof} for the proof. Observe that the descent of $\mathbb{E}\left[ f (\avgprm^{t+1}) \right]$ depends on the consensus error $\expec{\|\Prm\|_{\wtk}^2}$ which can be further bounded as:

\begin{lemma} \label{lemma:consensus} 
Under Assumptions \ref{assm:lip}--\ref{assm:graph_var} and the step size conditions $\alpha \le \frac{\rrhomin}{16 \rrhomax \theta } \min\{\frac{1}{ \rrhomax}, \frac{1}{\sigma_A^2}\}$, $\theta \ge \frac{12 L}{\rrhomin}$, then
\begin{equation}
    \begin{aligned}
	&\expec{ \| \Prm^{t+1} \|_{\wtk}^2} \leq \left( 1- (3/2) \alpha\theta \rrhomin \right) \expec{\|\Prm^t \|^2_{ \wtk }} \\
	& \quad + 3 \expec{ \|{\bf v}^t\|_{\wtk}^2 } - 2\expec{ \dotp{ \Prm^t}{ {\bf v}^t }_{ \left({\bf I} - \alpha\theta \wta^\top \rxii \wta \right)\wtk } } \\
    & \quad + \frac{9 \alpha \theta \rhomax^2}{ \rrhomin } \expec{ \| \hat{\Prm}^t-\Prm^t \|_{\wtk}^2} + \alpha^2\sum_{i=1}^n \sigma_i^2 .
	\end{aligned}
\end{equation}
\end{lemma}
\noindent See Appendix \ref{app:lemma_consensus_proof} for the proof.
The above lemma reveals that the consensus error $\expec{\|\Prm^t\|^2_{\wtk}}$ depends on $\expec{\|{\bf v}^t\|^2_{\wtk}}$, the weighted inner product between $\Prm^t,{\bf v}^t$, and the tracking error $\expec{ \| \hat{\Prm}^t-\Prm^t \|^2}$. The latter terms admit the following bounds:
    \begin{lemma} \label{lemma:dual_err}
        Under Assumption \ref{assm:lip}--\ref{assm:rand-graph} and the step size condition $\alpha \leq 1/4$,
        for any constant $\c > 0$, it follows that
            \begin{align*}
                &\expec{ \left\| {\bf v}^{t+1} \right\|_{\wtq + \c {\wtk}}^2} \leq (1+2\alpha)\expec{\left\| {\bf v}^t \right\|_{\wtq + \c \wtk}^2}  \\
                &\quad + (4\alpha^2\eta^2\rhomax^2 + 6\alpha^3 L^4)(\rrhomin^{-1}+\c) \expec{ \| \Prm^t \|_{\wtk}^2 } \\
                &\quad + 2\alpha\eta \expec{\dotp{{\bf v}^t}{ \Prm^t}_{\wtk+\c \wta^\top \rxii  \wta}} \\
                &\quad +   2\alpha\eta^2 \rhomax^2 (\rrhomin^{-1} + \c) \expec{\left\| \hat{\Prm}^t -\Prm^t \right\|_{\wtk}^2}    \\
                &\quad + 6\alpha^3 n L^2(\rrhomin^{-1} + \c) \left\{  \frac{1}{n^2}  \sum_{i=1}^n  \sigma_i^2  + \expec{\left\| \nabla f(\avgprm^{t}) \right\|^2}  \right\}.
            \end{align*}
    \end{lemma}
    \noindent See Appendix \ref{app:lemma_dualerr_proof} for the proof.

    \begin{lemma} \label{lemma:xv_inner} 
        Under Assumption \ref{assm:lip}--\ref{assm:rand-graph}, and the step size conditions
        $\alpha\le\min\{\frac{1}{12},\frac{\rrhomin}{2\rrhomax^2 \theta},\frac{\rrhomin}{2\sigma_{A}^2\rrhomax \theta}\}$, $\theta \ge \max\{\frac{4L^2}{\rrhomin},\frac{2\eta^2\rrhomax^2}{\rrhomin},\eta\}$, it follows that
            \begin{align*}
                &\expec{\dotp{\Prm^{t+1}}{\mathbf{v}^{t+1}}_{\wtk}} \leq 
                \expec{\dotp{\Prm^t}{\mathbf{v}^t}_{\wtk-(\alpha\theta+\alpha\eta)\wta^\top \rxii \wta }} \\
                &\quad - \frac{1}{8} \expec{ \| {\bf v}^t \|^2_{\wtk}}  +\frac{3\alpha}{2}\expec{\|\Prm^t\|^2_{\wtk}}   \\
                &\quad + \frac{1 }{2} ( \alpha\eta^2 \rrhomax^2 + 5 \alpha^2 \theta^2 \rhomax^2 ) \expec{ \| \hat{\Prm}^t-\Prm^t \|_{\wtk}^2} \\
                &\quad + \frac{9}{2}\alpha^3 n L^2 \expec{\left\| \nabla f(\avgprm^{t}) \right\|^2} +(\frac{9\alpha^3L^2}{2n}+\frac{3\alpha^3}{2})\sum_{i=1}^n  \sigma_i^2.
            \end{align*}
    \end{lemma}
    \noindent See Appendix \ref{app:lemma_xvinner_proof} for the proof.
    \update{Lemmas \ref{lemma:dual_err} and \ref{lemma:xv_inner} relate the concerned terms in two recursive inequalities. Importantly, we observe that the constant terms related to the noise variance are weighted by ${\cal O}(\alpha^3)$, suggesting that they can be effectively controlled by the step sizes.}
    \begin{lemma} \label{lemma:xhat-err} 
    Under Assumptions \ref{assm:compress}, \ref{assm:f_var}--\ref{assm:graph_var} and the step size condition $\alpha\le\sqrt{\gamma\delta/(8\rrhomax^2\theta^2)}$ and $\gamma \le \update{\frac{1}{r}}$, then
            \begin{align*}
                &\expec{\| \hat{\Prm}^{t+1} - \Prm^{t+1} \|^2} \leq (1-\frac{\update{r}\gamma\delta}{4})\expec{ \| \hat{\Prm}^t-\Prm^t \|^2} \\
                &\quad + \frac{4}{\update{r}\gamma\delta}[\alpha^2\theta^2(\rrhomax^2+\sigma_{A}^2\rrhomax/2)+\alpha^2L^2] \expec{\left\| \Prm^t \right\|^2_{\wtk} } \\
                &\quad + \frac{4\alpha\theta}{\update{r}\gamma\delta} \expec{ \dotp{ \Prm^t}{ {\bf v}^t }_{\wta^\top \rxii \wta } } 
                + \frac{4}{\update{r}\gamma\delta} \expec{ \|{\bf v}^t\|^2_{\wtk}} \\
                &\quad + \frac{4\alpha^2}{\update{r}\gamma\delta} \expec{\left\| \nabla f(\avgprm^{t}) \right\|^2 } + \frac{2\alpha^2}{\update{r}\gamma\delta} \sum_{i=1}^n \sigma_i^2 +\gamma^2\sigma_{\xi}^2.
            \end{align*}
    \end{lemma}
\noindent See Appendix \ref{app:lemma_xhaterr_proof} for the proof. \update{Especially, Lemma~\ref{lemma:xhat-err} also shows that the effect of communication noise can be controlled by the lower-level step size $\gamma$.} 

To upper bound $\|\nabla f(\avgprm^{t})\|^2$ alongside the miscellaneous coupling terms, we construct a potential function of the four error quantities. For some constants $\a,\b,\c,\d, \e > 0$ to be determined later, we define the potential function $F_t$ as
\begin{equation} \label{eq:Ft}
		\begin{aligned}F_{t}&=f(\avgprm^{t})+\a\|\Prm^{t}\|_{{\widetilde{\mathbf{K}}}}^{2}+\b\|\mathbf{v}^{t}\|_{{\widetilde{\mathbf{Q}}+\c\widetilde{\mathbf{K}}}}^{2}\\&+\d\left\langle\Prm^{t}\mid\mathbf{v}^{t}\right\rangle_{{\widetilde{\mathbf{K}}}} +\e\| \hat{\Prm}^t-\Prm^t \|^2 .
		\end{aligned}
\end{equation}
We observe that:


\begin{lemma} \label{lemma:para_choose} Under Assumptions \ref{assm:lip}--\ref{assm:graph_var}, \ref{assm:compress}. Set
\begin{equation} \label{eq:abcd_choice} 
    \begin{aligned}
        &\b = \a \cdot \frac{1}{\alpha \eta}, \quad \c = \frac{(\alpha\theta+\alpha\eta)\d - 2\alpha \theta \a- \e \frac{4}{\update{r}\gamma\delta}\alpha\theta}{2 \alpha\eta \b},\\
        &\d = \frac{1024\rhomax}{\update{r}\gamma\delta\rrhomin} \a,\quad \e = 8\rhomax \rrhomin^{-1} \a,
    \end{aligned} 
\end{equation}
for some $\a > 0$.
Then, for $\eta = \frac{\update{r}\gamma\delta}{8\rhomax}, \theta \ge \theta_{lb}, \alpha\le\alpha_{ub}$ [cf.~\eqref{eq:stepsize-thm}], $\gamma \le \update{\frac{1}{r}}$, it holds that $F_t \ge f(\avgprm^{t}) \ge f^* > -\infty$, and
\begin{equation} \label{eq:lemma_potential}
    \begin{aligned}
    & \expec{F_{t+1}} \leq \expec{F_t} -\frac{1}{16}\alpha \expec{\left\| \nabla f(\avgprm^{t}) \right\|^2} \\
    & \quad + \alpha^2\omega_{\sigma} \bar{\sigma}^2 - \frac{1}{4} \alpha\theta\rrhomin \a  \expec{\|\Prm^t\|_{\wtk}^2} +8\a \gamma^2\sigma_{\xi}^2\update{\frac{\rhomax}{\rrhomin}} .
\end{aligned}
\end{equation}
such that 
\begin{align*}
&\omega_{\sigma} = \frac{L}{2n} + \a \Big(n +\frac{12L^2}{\eta\rrhomin} +\frac{3\rrhomin\alpha\theta}{16}(3L^2+n)+\frac{16n\rhomax}{\gamma\delta\rrhomin} \Big).
\end{align*}
\end{lemma}
\noindent The lemma is obtained through satisfying the step size conditions in the previous lemmas \update{while optimizing the constants $\a, \b, \c, \d, \e$ to extract the best convergence rates on $\mathbb{E} [ \| \nabla f(\avgprm^{t}) \|^2 ]$}. See Appendix \ref{app:lemma_parachoose_proof} for the proof.

Summing up the inequality \eqref{eq:lemma_potential} from $t=0 $ to $t=T-1 $ and divide both sides by $\alpha T$ gives us
\begin{align}
    &\frac{1}{4T} \sum_{t=0}^{T-1} \left\{ \frac{1}{4} \expec{\left\| \nabla f(\avgprm^{t}) \right\|^2}+  \theta\rrhomin \a \expec{\|\Prm^t\|_{\wtk}^2} \right\} \notag \\
    & \leq \frac{ \expec{F_0} - \expec{F_{T}}}{ \alpha T} + \alpha \omega_{\sigma} \bar{\sigma}^2 +8\a \frac{\gamma^2\sigma_{\xi}^2}{\alpha}. \label{proof:thm_last2}
\end{align}
Reshuffling terms yields the conclusions of the theorem.

\section{Numerical Experiments} \label{sec:Numerical Experiments}
This section demonstrates the effectiveness of {\algoname} on practical problems through numerical experiments. As we aim at testing the performance of {\algoname} in unreliable networks, we also evaluate the \emph{total number of bits transmitted across the network}. 
For benchmarking purposes, throughout this section, we focus only on decentralized algorithms that support compressed message exchanges. \update{Specifically, we compare CHOCO-SGD \cite{koloskova2019decentralized}, DoCoM \cite{yau2022docom}, DIMIX \cite{reisizadeh2023dimix}, LEAD \cite{liu2020linear} and CP-SGD \cite{xie2024communication}
which support different types of compression operators, as well as FSPDA-SA \cite{yau2024fully} that only utilizes sparsification and DSGD \cite{koloskova2020unified} as an uncompressed random graph baseline.}
The hyperparameters of the tested algorithms are hand tuned via a grid search on the magnitude to achieve the lowest gradient norm $\| \nabla f( \avgprm^T ) \|^2$ after $T$ iterations.  
\update{Our experiments are run on servers of Intel Xeon Gold 6148 CPU (for Sec.~\ref{sec:linear_reg}, \ref{sec:sigmoid}) and 8 $\times$ NVIDIA RTX 3090 GPU (for Sec.~\ref{sec:imagenet}).
Our implementation utilizes PyTorch with MPI communication, which supports pairwise communication between any pair of processes to simulate communication on random networks. For noisy communication, we emulate real-world network noise by artificially adding random noise information into the messages sent within MPI.}

We consider \update{two} types of compression operators $Q(\cdot; \xi_q)$. In addition to random quantization in \eqref{eq:qsgd} with $s = 2^4$ levels, top-$k$ 
sparsification keeps the $k$ coordinates with the highest magnitude.
{Note that each compressed message with random quantization takes $d (\log_2(s)+1) + 32$ bits to transmit, while the sparsified message takes $64k$ bits to transmit.} Lastly, $G = (V,E)$ is a graph with $n$ agents, \update{generated as an ER graph with $p=0.5$ connectivity}. At each iteration, the {\algoname}, DIMIX and FSPDA-SA algorithms draw a random subgraph $G(\xi_a)$ \emph{with only 1 active edge} from $E$, while \update{CHOCO-SGD, DoCoM, CP-SGD and LEAD} take a broadcasting subgraph design where $G(\xi_a)$ is formed by taking the edges incident to \emph{only 1 randomly selected agent}. Notice that the latter two algorithms are only shown to converge under such restrictive type of time-varying communication graphs \cite{koloskova2019decentralized}. 
\vspace{.1cm}

\update{
\subsection{Linear Regression on Synthetic Data} \label{sec:linear_reg}
Our first numerical experiment considers a linear regression problem of $d=100$ dimensional features with
\begin{equation} \notag \textstyle 
    f_i(\prm) = \frac{1}{30}\sum_{j=1}^{30} \| \dotp{\prm}{{\bf z}_{i,j}} - \phi({\bf z}_{i,j}; \epsilon_{i,j}) \|^2,
\end{equation}
i.e., each agent holds $m_i=30$ samples. We consider a setting of $n=10$ agents. 
The local features vectors are ${\bf z}_{i,j} \sim {\rm Uniform}(0, 1)^{100}$ and local labels are $\phi({\bf z}_{i,j}; \epsilon_{i,j}) = \dotp{\hat{\prm}}{{\bf z}_{i,j}} + \epsilon_{i,j}$ such that $\epsilon_{i,j} \sim {\rm Uniform}(0, 0.1)$ and the ground truth parameter $\hat{\prm}$ satisfies $\hat{\prm} \sim {\rm Uniform}(-1, 1)^{100}$. Note this is the same problem setup as \cite{reisizadeh2023dimix}.
Fig.~\ref{fig:dimix_syn_exp} compares {\algoname} with the benchmark algorithms in the case with noiseless communication, i.e., $\sigma_{\xi} = 0$. As observed, with a 4-bit randomized quantizer, {\algoname} achieves the fastest convergence in terms of the number of bits transmitted. 
However, we notice that DiMIX only converges to a sub-optimal solution. We speculate that this is due to the slow (practical) convergence behavior of algorithms that rely on the DSGD framework.
}

\begin{figure}[t]
	\centering
    \includegraphics[width=0.9\linewidth]{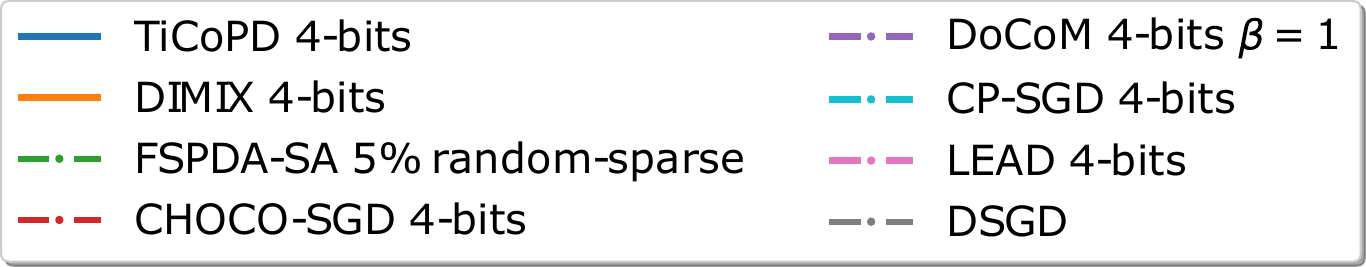} \\[.5em]
    \includegraphics[width=0.425\linewidth]{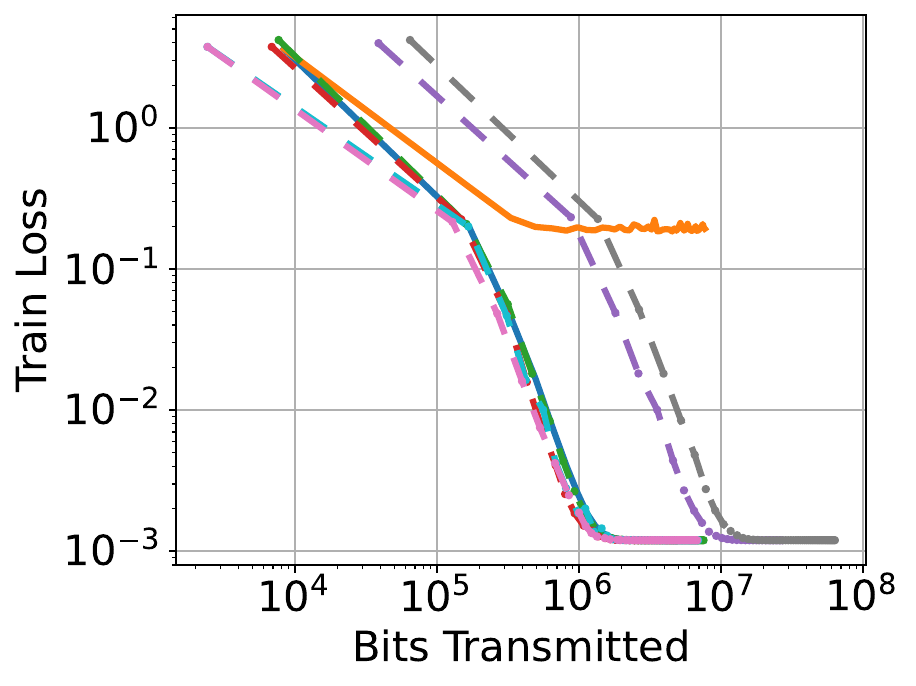}
	\includegraphics[width=0.425\linewidth]{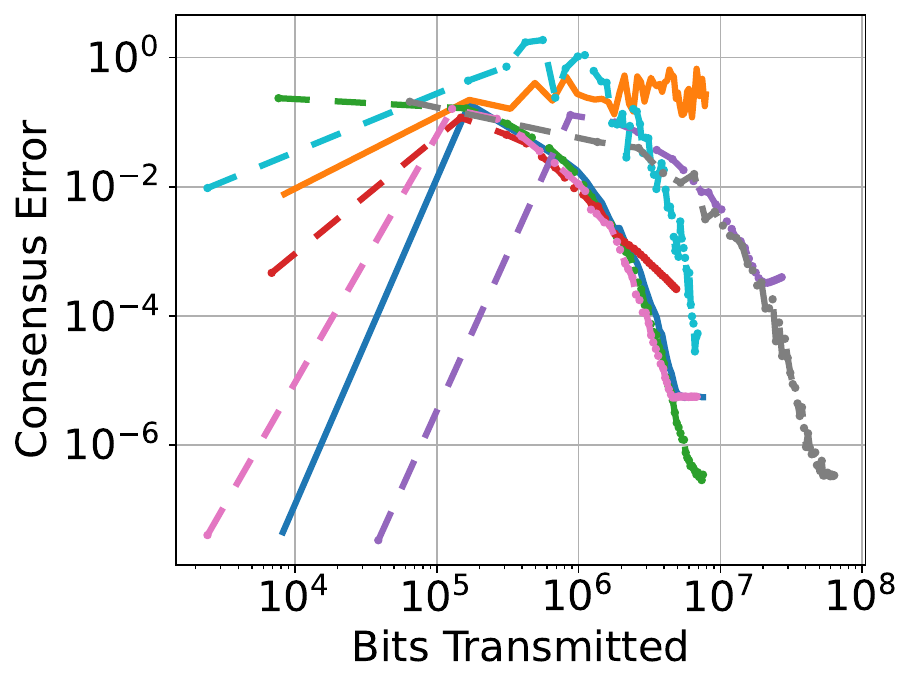}\\
	\includegraphics[width=0.425\linewidth]{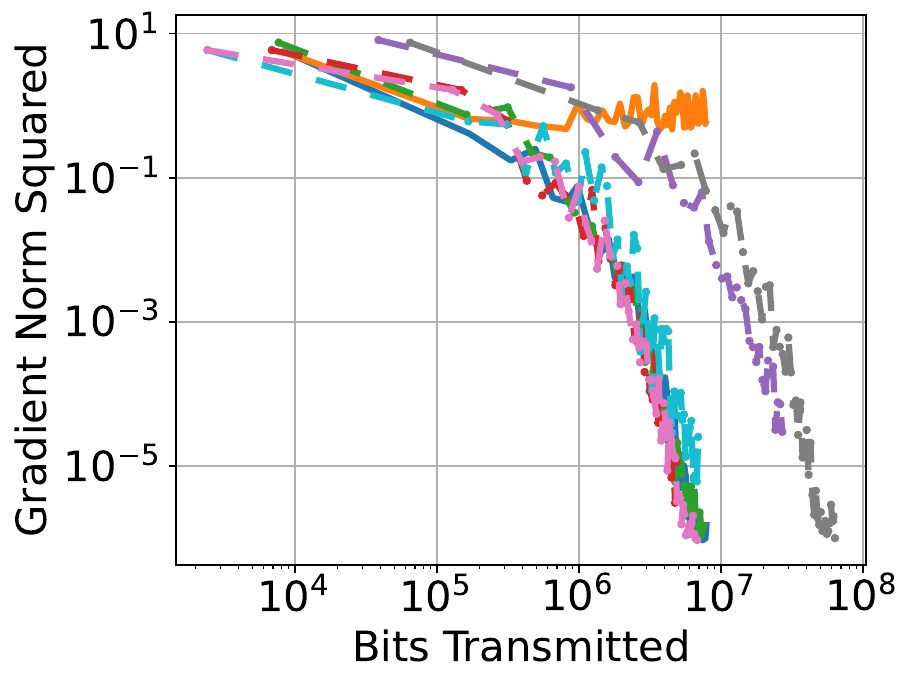}
    \vspace{-.2cm}
	\caption{Convergence over noiseless communication channel with random edge activation \update{on the linear regression problem}. 
    }\vspace{-.2cm}
	\label{fig:dimix_syn_exp}
\end{figure}

\update{\subsection{Non-convex Sigmoid Loss on Synthetic Data}\label{sec:sigmoid}} 
Our second set of numerical experiments considers learning a linear model from a synthetic dataset to simulate a controlled decentralized learning environment with heterogeneous data. We consider a set of $n=10$ agents, where each agent $i$ holds a set of $100$ observations $\{{\bf z}_{i,j}, \phi({\bf z}_{i,j})\}_{j \in [100]}$. The feature vectors ${\bf z}_{i,j} \in \mathbb{R}^{100}$ are generated as ${\bf z}_{i,j} \sim \mathcal{N}({\bf m}_i , 0.5 {\bf I})$, where ${\bf m}_i \sim {\rm Uniform}\left(\frac{-n/2 + i-1}{n/2}, \frac{-n/2 + i}{n/2}\right)^{100}$.  The labels $\phi({\bf z}_{i,j}) \in [10]$ are determined through the map $\phi({\bf z}) = \arg \max_{k} \{ {\bf z}^\top {\bf x}^{\rm truth}_k \}$, where ${\bf x}^{\rm truth}_k \sim {\rm Uniform}(-1,1)^{100}$, $k \in [10]$ is the ground truth model for label $k$. To learn the linear model $\prm^{\rm truth} = ( \prm_1^{\rm truth} , \ldots, \prm_{10}^{\rm truth} ) \in \mathbb{R}^{1000}$, we consider the local objective function $f_i(\prm)$ as the sigmoid loss:
\begin{align} \notag
\sum_{k=1}^{10} \Big( \frac{1}{100} \sum_{j=1}^{100} {\rm sigmoid} ( \mathds{1}_{ \{ \phi({\bf z}_{i,j}) = k \} } \, \prm_{k}^\top {{\bf z}_{i,j}} ) + \frac{10^{-4}}{2}  \| \prm_{k} \|^2 \Big),
\end{align}
where ${\rm sigmoid}(y) = (1 + e^{-y})^{-1}$ for any $y \in \mathbb{R}$ and $\mathds{1}_{ \{ \cdot \} } \in \{ \pm 1\}$ is the indicator function.

\begin{figure}[t]
	\centering
    \includegraphics[width=0.9\linewidth]{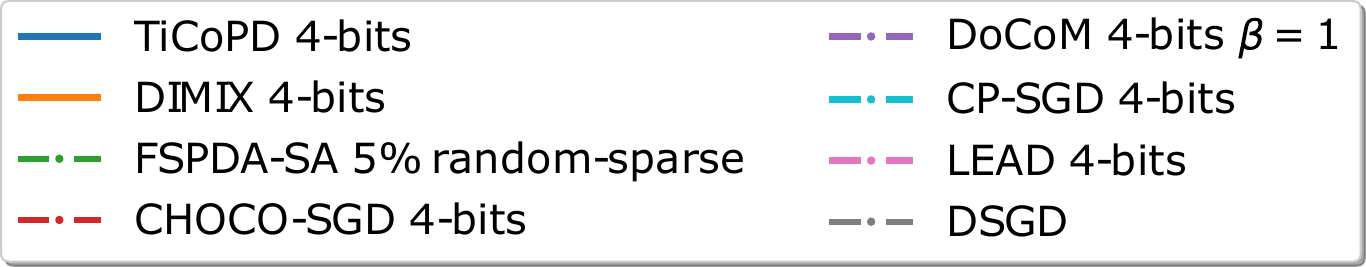} \\[.5em]
    \includegraphics[width=0.425\linewidth]{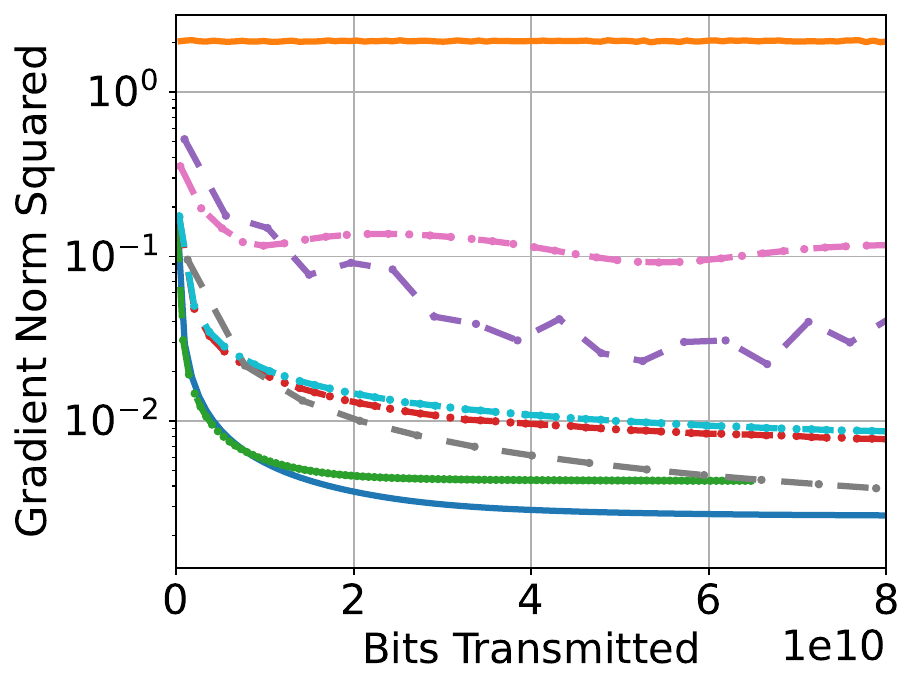}
	\includegraphics[width=0.425\linewidth]{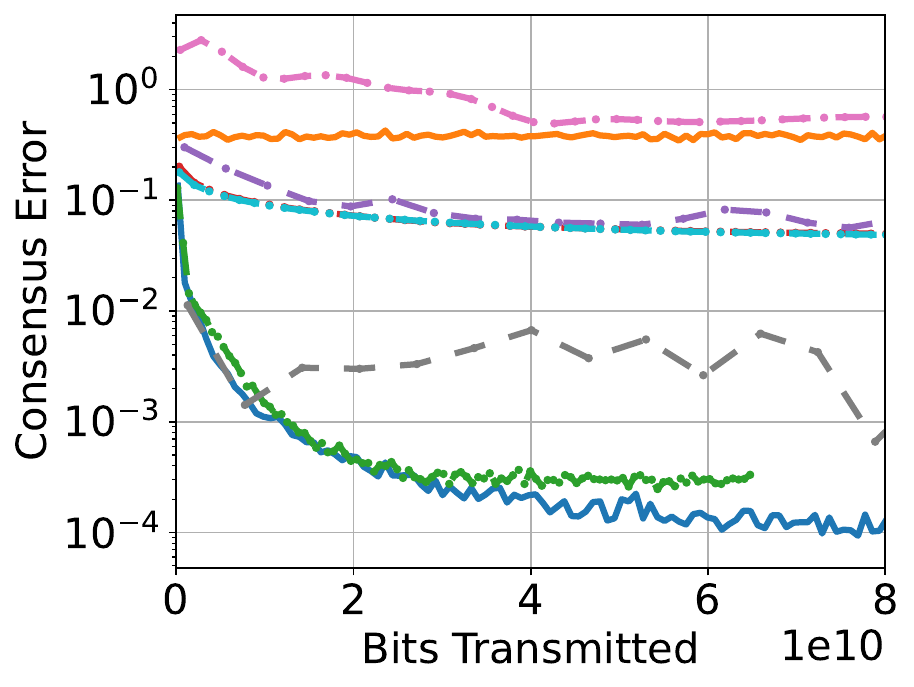}\\
	\includegraphics[width=0.425\linewidth]{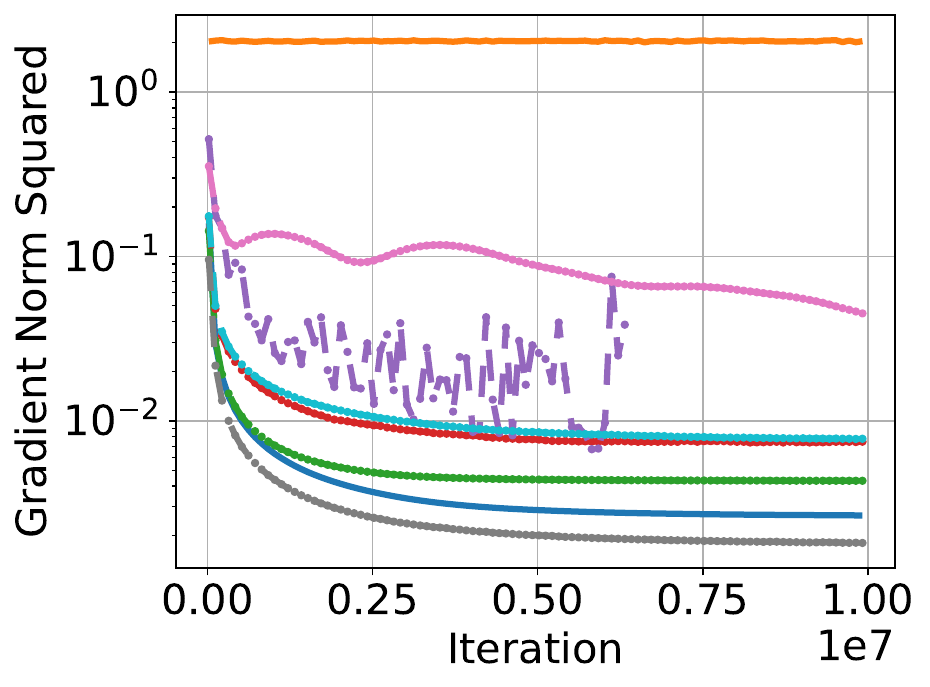}
	\includegraphics[width=0.425\linewidth]{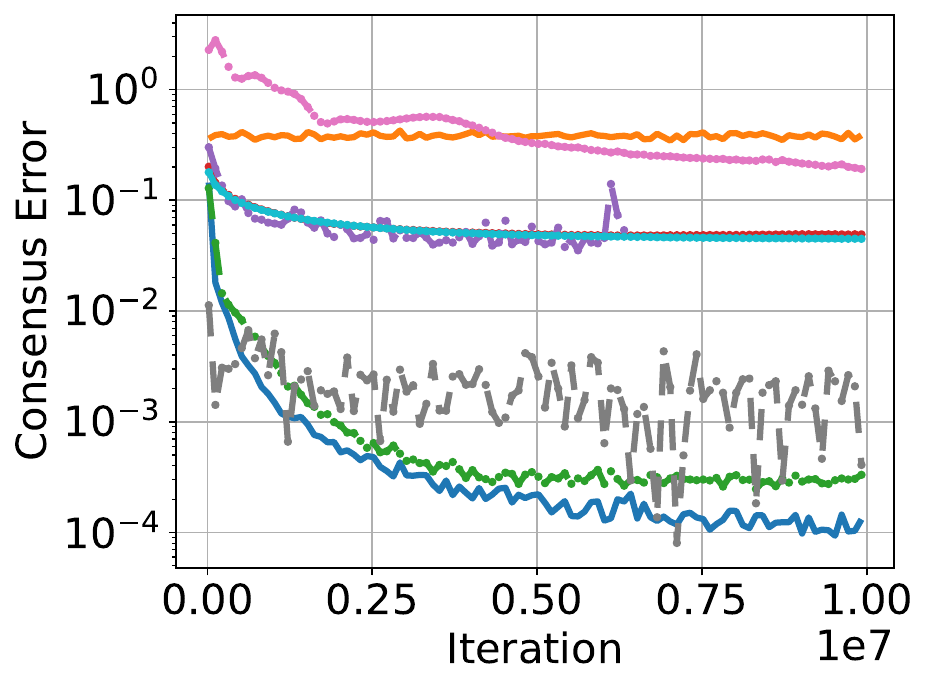}
    \vspace{-.2cm}
	\caption{Convergence over noiseless communication channel with random edge activation \update{for the non-convex sigmoid loss minimization}. (Top) against the total number of bits transmitted over the network. (Bottom) against the iteration number.}\vspace{-.2cm}
	\label{fig:syn_exp}
\end{figure}

Fig.~\ref{fig:syn_exp} compares the performance of {\algoname} with benchmarked algorithms when the communication network is \emph{noiseless}, i.e., $\sigma_{\xi} = 0$. For {\algoname}, we have used $\eta = 0.005, \gamma = 1, \theta = 10^2, \alpha = 10^{-4}$. 
\update{Observe that {\algoname} finds the best near-stationary solution among other communication compressed baselines, and achieves the smallest consensus error.}


\begin{figure}[t]
	\centering
    \includegraphics[width=0.99\linewidth]{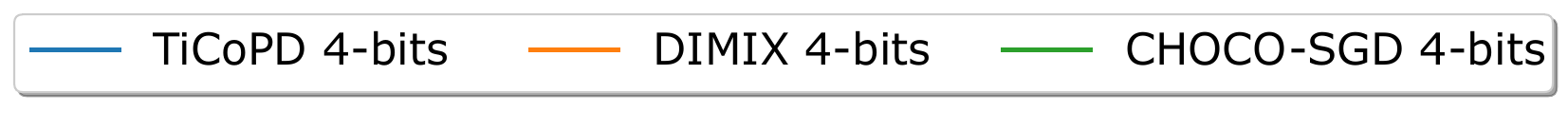} \\[.5em]
    \includegraphics[width=0.425\linewidth]{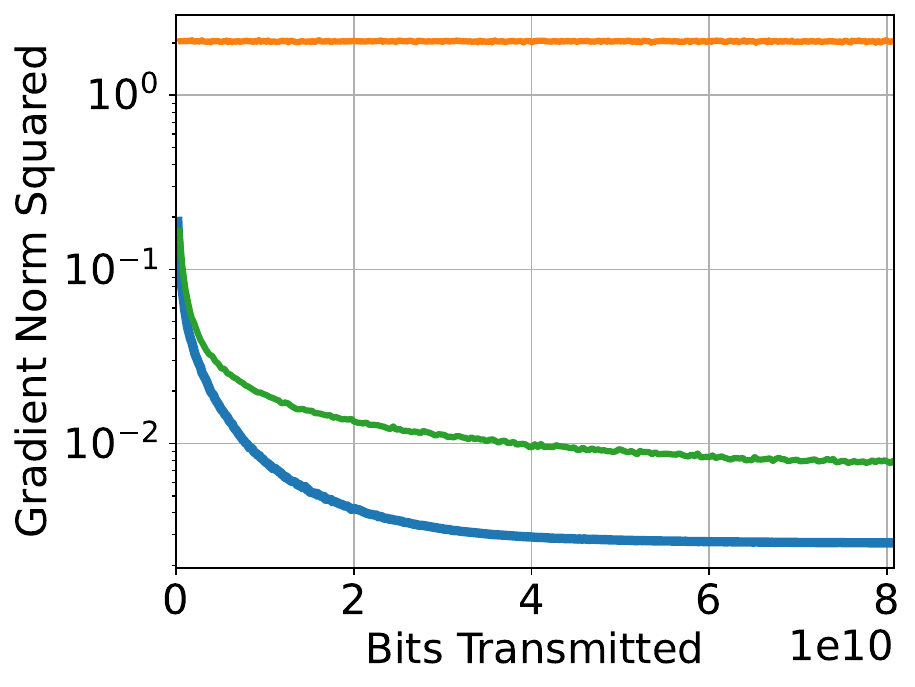}
	\includegraphics[width=0.425\linewidth]{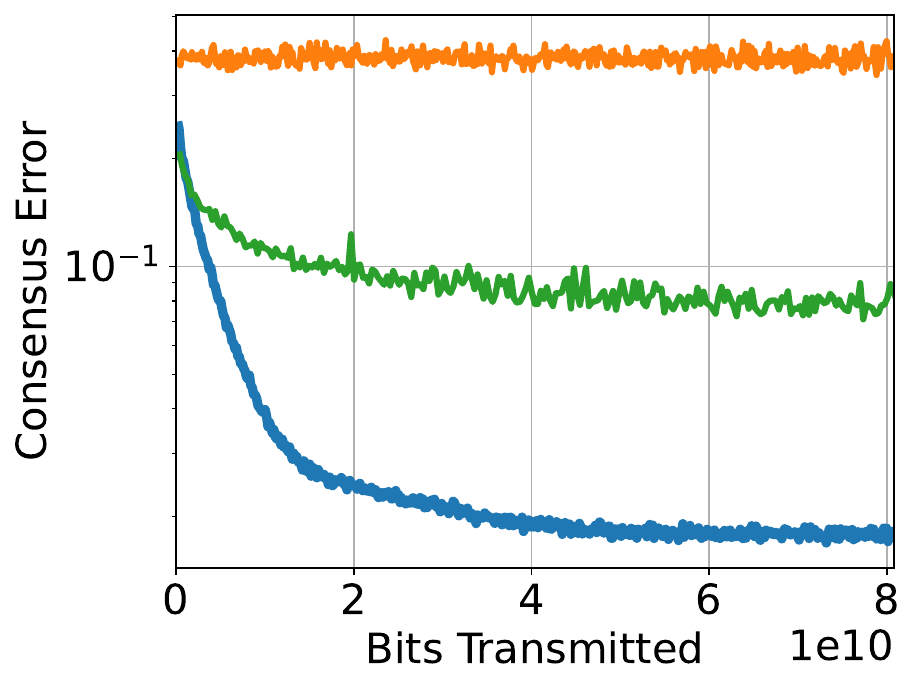}\\
	\includegraphics[width=0.425\linewidth]{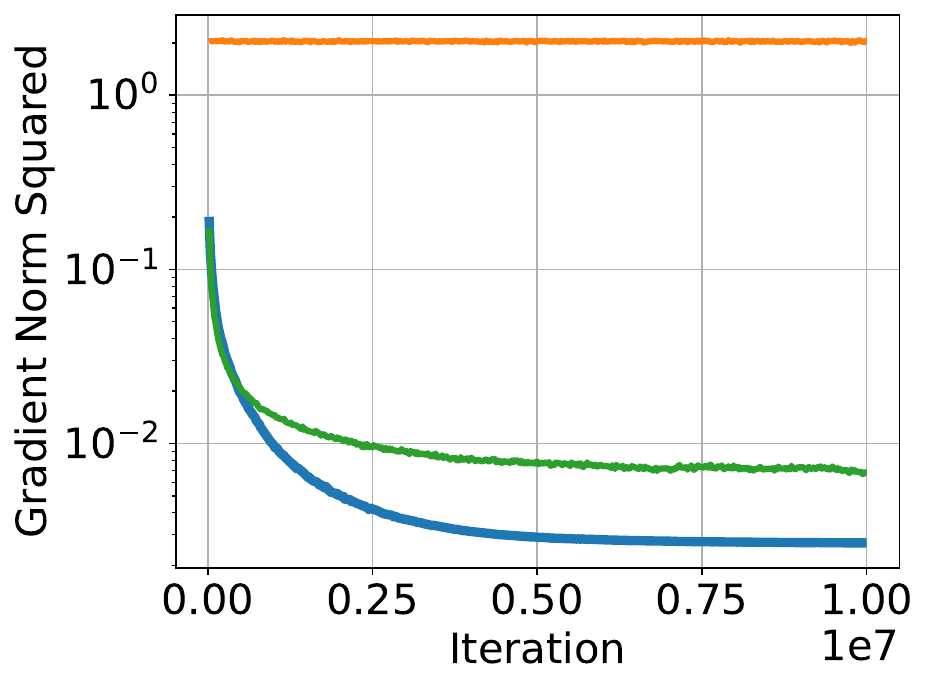}
	\includegraphics[width=0.425\linewidth]{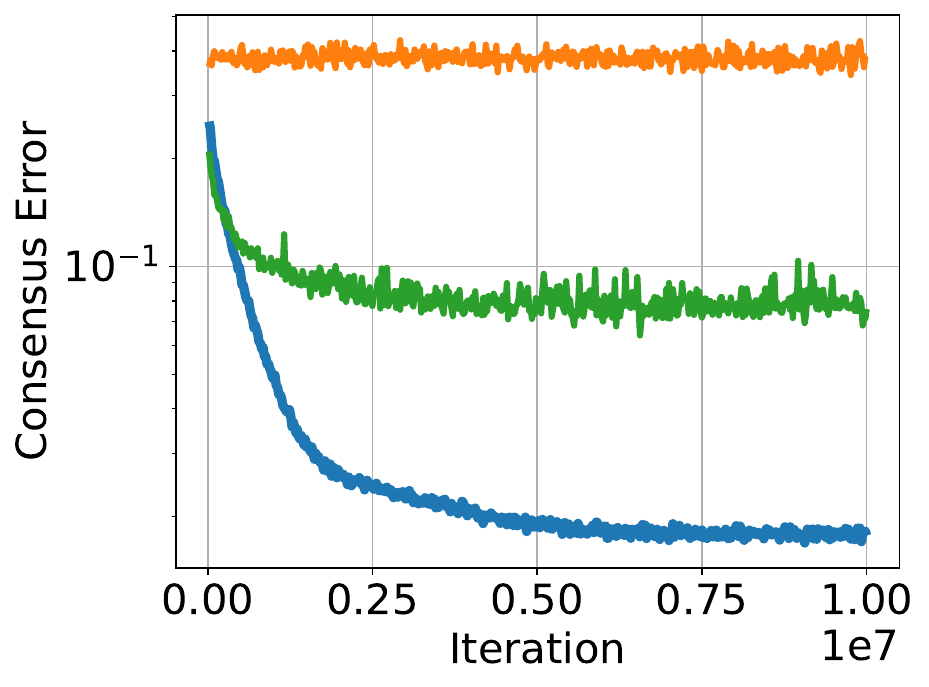}
    \vspace{-.2cm}
	\caption{Convergence of compressed algorithms over communication channel with random edge activation and \emph{additive communication noise} \update{on the non-sigmoid loss minimization problem}. (Top) against number of bits transmitted over the network. (Bottom) against iteration number.}\vspace{-.2cm}
	\label{fig:syn_exp_noise}
\end{figure}

\begin{figure}[t]
	\centering
	\includegraphics[width=0.8\linewidth]{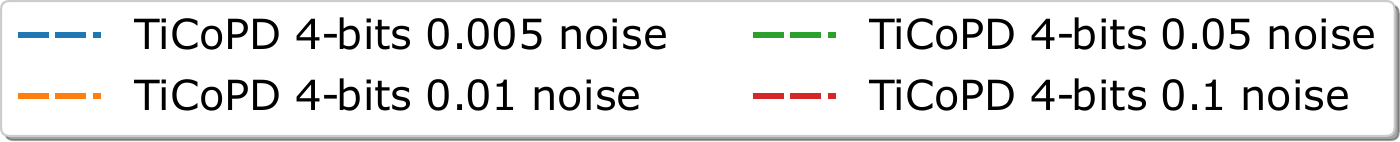} \\[.5em]
	\includegraphics[width=0.425\linewidth]{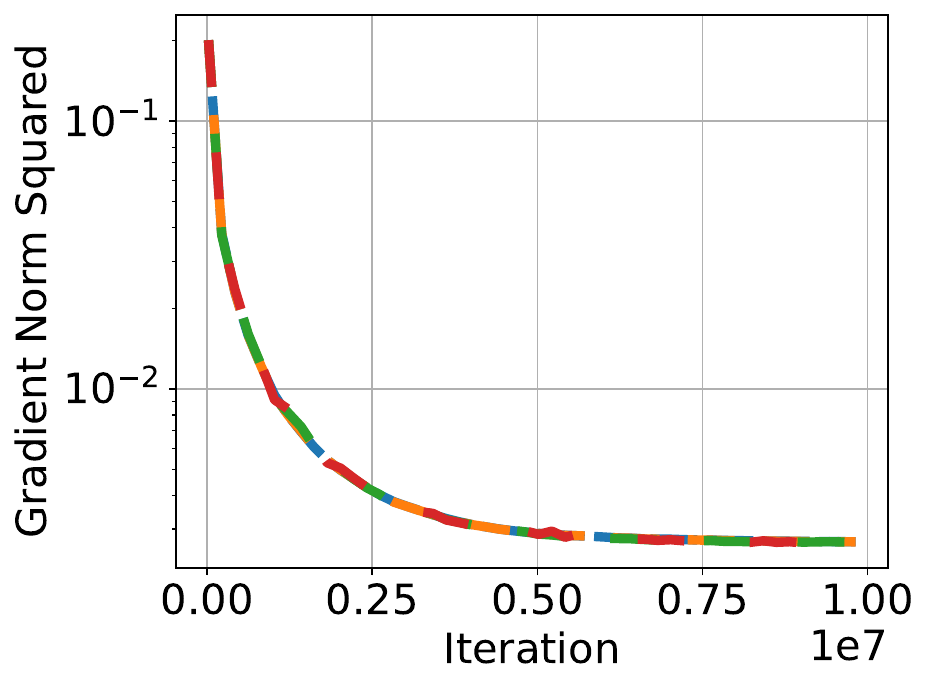}
	\includegraphics[width=0.425\linewidth]{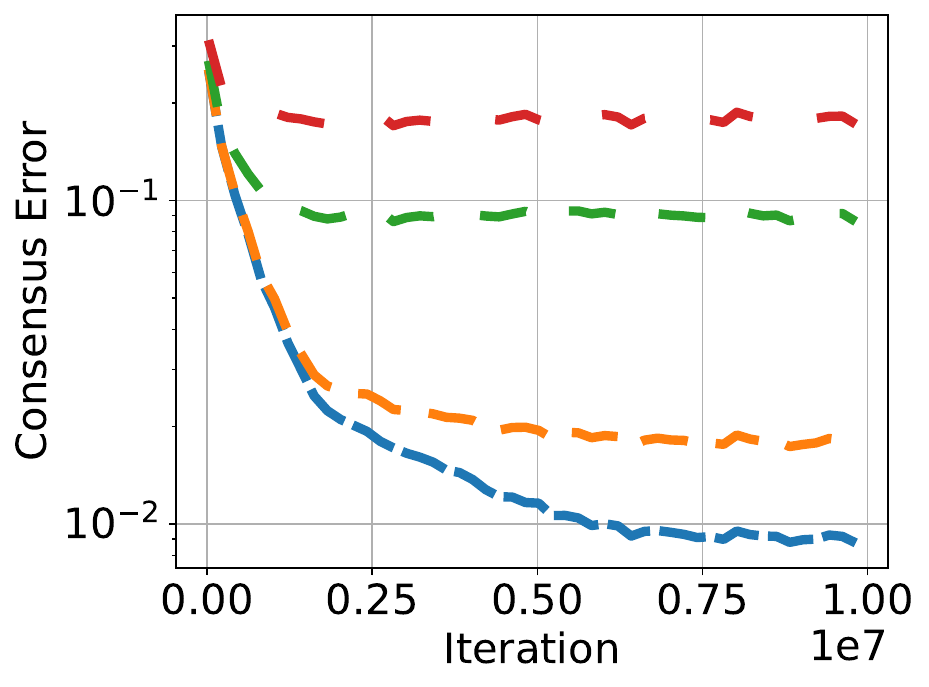}
    \vspace{-.2cm}
	\caption{Convergence of {\algoname} over communication channel with random edge activation and different levels of \emph{additive noise} \update{on the non-convex sigmoid loss minimization problem}.} 
	\label{fig:syn_exp_noiselevel}
\end{figure}

We next study the effects of noisy communication channel when $\sigma_{\xi} = 0.01$ on {\algoname} \update{in Fig.~\ref{fig:syn_exp_noise}}. Recall from the discussions following Theorem~\ref{theo:main_theorem}, the introduction of communication noise may slow down the convergence rate of {\algoname}, 
\update{yet the algorithm can still manage to find a solution within a neighborhood of a stationary point with radius proportional to $10^{-2}$ of \eqref{eq:origin_problem}.}

Fig.~\ref{fig:syn_exp_noise} compares the convergence behavior of {\algoname} to DIMIX and CHOCO-SGD. We note that under the noisy channel setting, it is not clear if CHOCO-SGD will converge, yet with a carefully tuned stepsize, DIMIX can theoretically converge to a stationary and consensual solution of \eqref{eq:origin_problem} \cite{reisizadeh2023dimix}. From the figure, we observe that {\algoname} converges as predicted by our theorem. Meanwhile, CHOCO-SGD fails to find a consensual solution and DIMIX fails to find a stationary solution.
\update{For {\algoname}, we set $\gamma=0.1,\theta=100, \eta = 10^{-4}$, $\alpha = 10^{-4}$.}
Lastly, Fig.~\ref{fig:syn_exp_noiselevel} compares the performance of {\algoname} at different levels of communication noise $\sigma_{\xi} \in \{ 0.005, 0.01, 0.05, 0.1 \}$. As predicted by Theorem~\ref{theo:main_theorem}, varying the communication noise level does not affect the convergence rates of {\algoname}, but it affects the magnitude of the dominant term in the stationarity and consensual error of the solutions found. 

\begin{figure}[t]
    \centering
    \includegraphics[width=0.9\linewidth]{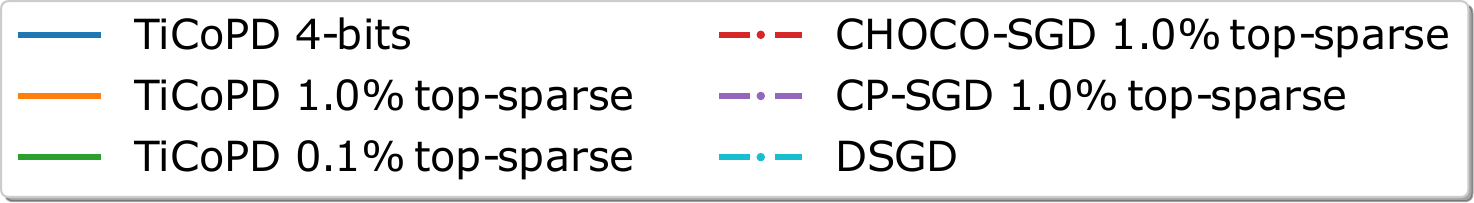} \\[.25cm]
    \includegraphics[width=0.425\linewidth]{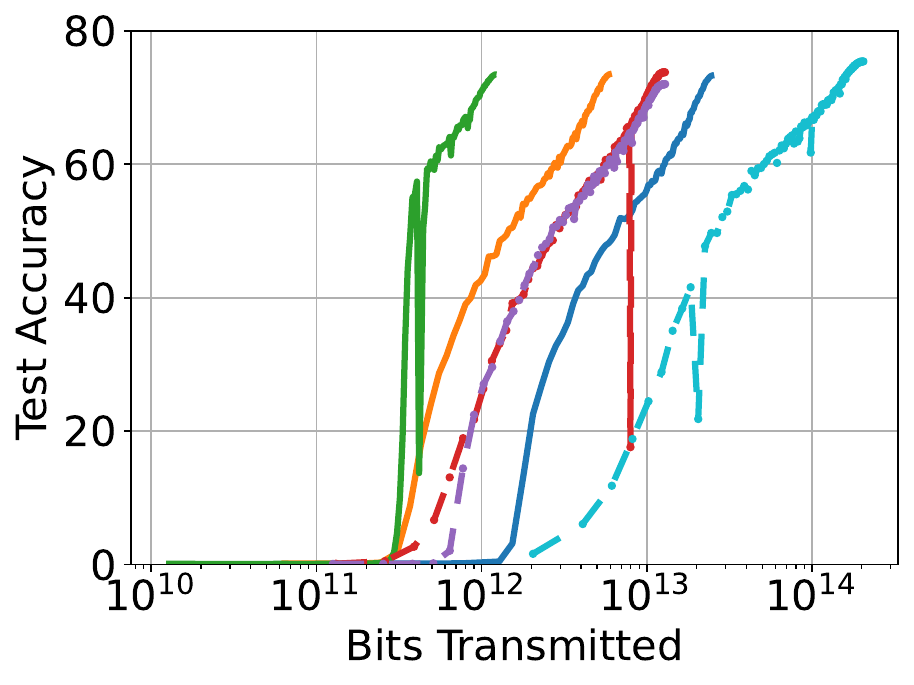}
    \includegraphics[width=0.425\linewidth]{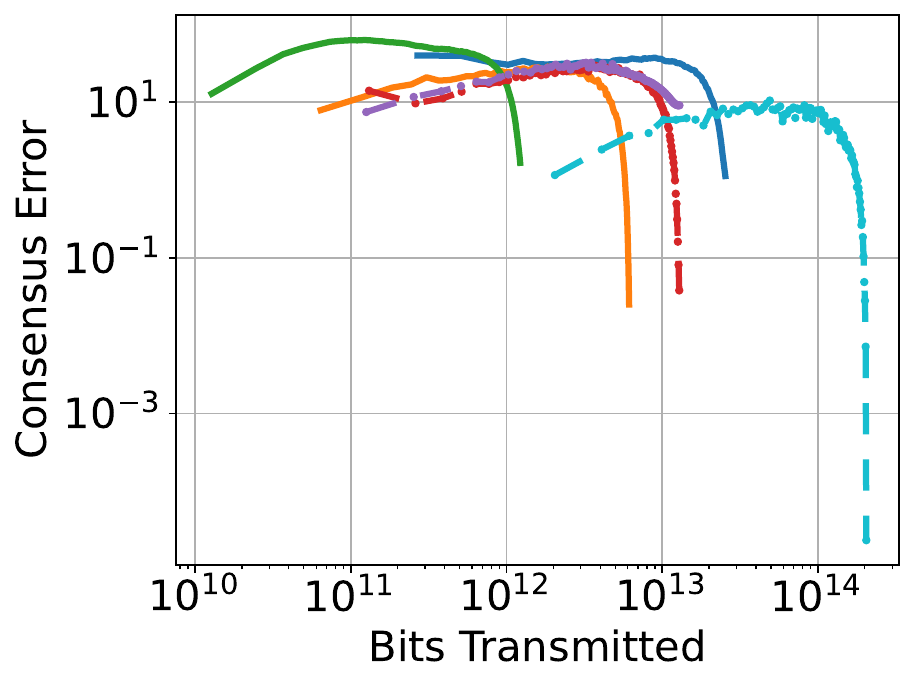} \\
    \includegraphics[width=0.425\linewidth]{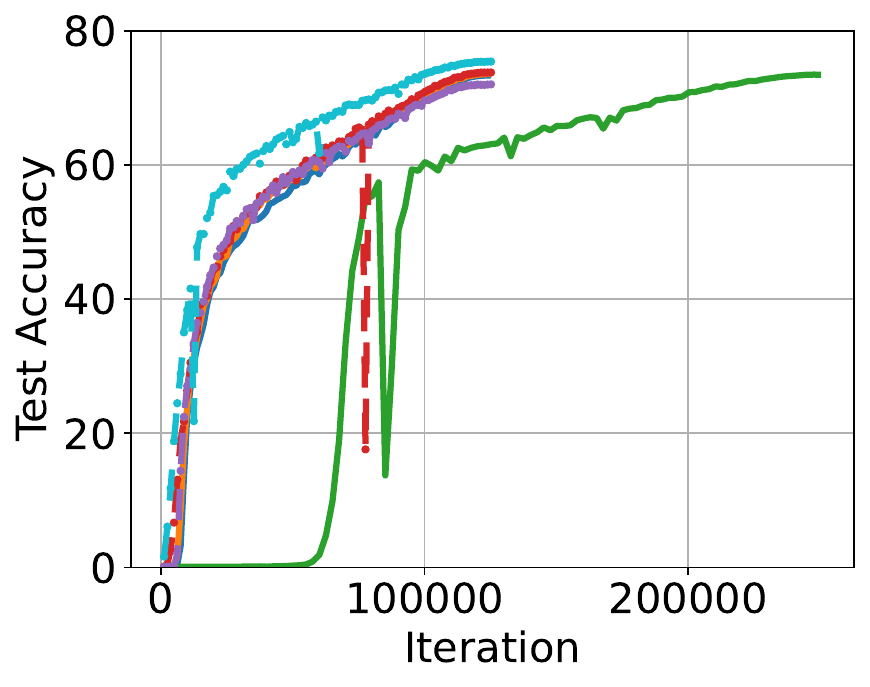} 
    \includegraphics[width=0.425\linewidth]{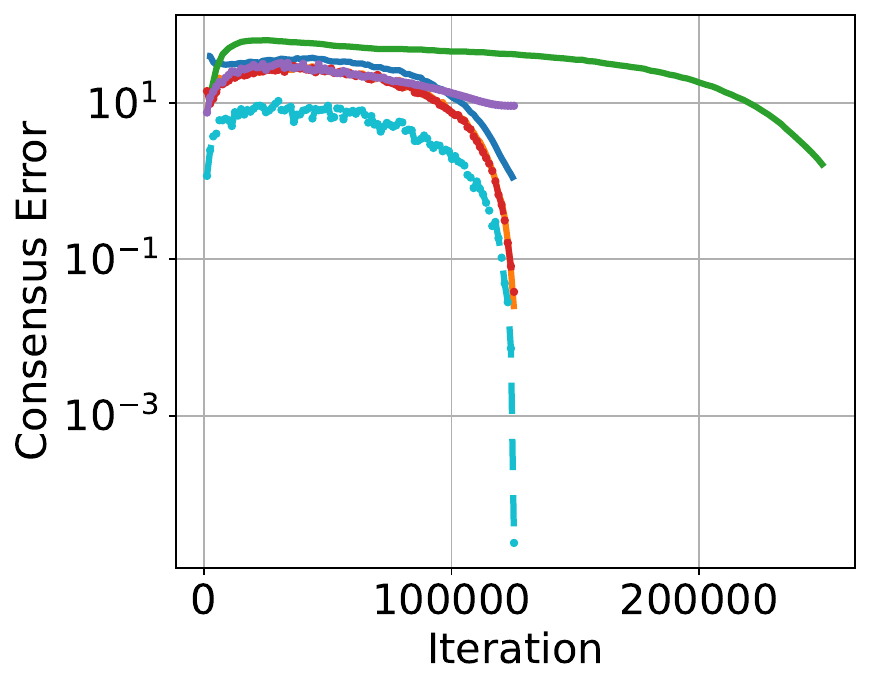}\vspace{-.2cm}
    \caption{Convergence of the training problem of ResNet-50 for Imagenet classification. The performance is evaluated on the network-averaged model $\avgprm$. (Top) against number of bits transmitted over the network. (Bottom) against iteration number.} 
    \label{fig:imagenet}
\end{figure}

\begin{table}[t] 
\centering
\begin{tabular}{l|cc}
\toprule
  Compressor    & Forward+Backward & Comm. \\
\midrule
4-bits (RTX 3090)     & 0.252 s          & 0.089 s     \\
8-bits (RTX 3090)     & 0.253 s          & 0.220 s     \\
16-bits (RTX 3090)     & 0.253 s          & 0.256 s     \\
\midrule
1\% sparse (RTX 3090) & 0.253  s & 0.011 s   \\
\bottomrule 
\end{tabular}\vspace{.1cm}
\caption{Average time of computation (i.e., forward and backward step) and communication per iteration on the Resnet-50 experiment with MPI backend.} \label{tab:runtime_stat}
\end{table}

\update{\subsection{Distributed Neural Network Training on Real Data}\label{sec:imagenet}} 
The \update{third} set of experiments considers the case of a deep neural network distributed training task. We consider learning a classifier model based on the ResNet-50 architecture \update{$\phi({\bf z}; \prm): \mathbb{R}^{256 \times 256 \times 3} \rightarrow \mathbb{R}^{1000}$ (with $d=2.56 \times 10^9$ parameters) on the {\tt ImageNet} dataset $\{ ({\bf z}_j, {\bf y}_{j}) \}_{j=1}^m$ for image features ${\bf z}_k \in \mathbb{R}^{256\times 256 \times 3}$ and image class index ${\bf y}_k \in [1000]$ (with $m = 1,281,168$ samples, split equally to $n=10$ agents with shuffling) and $f_i(\prm)$ as the standard cross entropy loss:
\begin{equation}
   f_i(\prm) = - \frac{1}{m_i} \sum_{j=1}^{m_i} \log \frac{\exp(\phi({\bf z}_{i,j}; \prm)_{{\bf y}_{i,j}})}{\sum_{k=1}^{1000} \exp(\phi({\bf z}_{i,j}; \prm)_{k})}
\end{equation}}Note that {\algoname} coupled with \update{1\% sparsification compression and 1-edge activated random communication graph offers a communication compression of up to 40$\times$ per iteration compared to uncompressed DSGD on 1-edge random graphs.}
Further savings can be achieved with a more aggressive compression scheme, e.g., with a $0.1\%$-sparsifier. 
Fig.~\ref{fig:imagenet} shows the convergence of {\algoname} and other baselines under different order of compression \update{on 1-edge random graph.}
As predicted by our theory, the network-averaged iterate of TiCoPD will converge to the same degree of error after the transient time, at the cost of only increasing the consensus error. 
\update{We adopt the same learning rate scheduling from \cite{loshchilov2016sgdr} which applies 5 epochs of linear warm-up followed by a cosine decay till the end of training. For TiCoPD with 4-bits quantization, we set $\gamma = 1, \theta=0.05, \eta = 10^{-8}, \max_t \alpha_t = 0.1$. For TiCoPD with 1\% top-sparse, we set $\gamma = 1, \theta=0.05, \eta = 10^{-7}, \max_t \alpha_t = 0.1$. For TiCoPD with 0.1\% top-sparse, we set $\gamma = 1, \theta=0.02, \eta = 10^{-13}, \max_t \alpha_t = 0.05$.}
\update{Lastly, Table~\ref{tab:runtime_stat} compares the average time of computation and communication for this experiment. We observe that more aggressive compression schemes can reduce the communication time.} 


\section{Conclusion}    \label{sec:Conclusion}
    This paper studies a communication efficient primal-dual algorithm for decentralized stochastic optimization with support for time-varying graphs and nonlinear compression schemes such as quantized message exchanges. Our key innovation is to incorporate majorization-minimization and two-timescale updates into the  augmented Lagrangian framework. In $T$ iterations, the resultant {\algoname} algorithm finds an ${\cal O}(T^{-\frac{1}{2}})$-stationary solution for smooth (possibly non-convex) problems in the absence of communication noise, and ${\cal O}(T^{-\frac{1}{3}})$-stationary solution in the presence of communication noise. We envisage that the proposed algorithmic framework can inspire the design of sophisticated decentralized algorithms such as algorithms with Nesterov acceleration. 
    \update{For future works, it is important to extend TiCoPD to constrained problems and incorporate acceleration techniques.}


\appendices
\section{Proof of Lemma \ref{lemma:descent}} \label{app:lemma_descent_proof} 
As ${\bf 1}^\top \wta^\top  = {\bf 0}$, the updates of $\avgprm$ satisfies:
\begin{equation}
	\begin{aligned}
	\avgprm^{t+1} 
	&= \avgprm^t   - \frac{\alpha}{n} \oneotimesT\nabla {\bf f}(\Prm^t;\xi^{t+1}) , \label{eq:avgprm_update}
\end{aligned}
\end{equation}
We have denoted the shorthand notations $\oneotimesT := {\bf 1}^\top \otimes {\bf I}$ and $\oneotimes := {\bf 1} \otimes {\bf I}$.  
Using Assumption \ref{assm:lip} and \eqref{eq:avgprm_update}, we get
	\begin{equation}
	    \begin{aligned}
		\mathbb{E}\left[ f(\avgprm^{t+1}) \right] & \le \mathbb{E}\left[ f(\avgprm^{t}) - \dotp{ \nabla f(\avgprm^{t})}{ \frac{\alpha}{n} \oneotimesT\nabla {\bf f}(\Prm^t) }   \right] \\
        &+ \frac{L}{2} \expec{ \left\|\frac{\alpha}{n} \oneotimesT\nabla {\bf f}(\Prm^t;\xi^{t+1}) \right\|^2} . \label{proof:lemma_desc_step1}
	\end{aligned}
	\end{equation}
	The second term of \eqref{proof:lemma_desc_step1} can be bounded as
	\begin{equation}
	    \begin{aligned}
		& -\expec{ \dotp{ \nabla f(\avgprm^{t})}{  \frac{\alpha}{n} \oneotimesT\nabla {\bf f}(\Prm^t) } }\\
		\le &-\frac{\alpha}{2}\expec{ \left\|\nabla f(\avgprm^{t}) \right\|^2 }+ \frac{\alpha}{2} \expec{ \left\| \nabla f(\avgprm^{t})-\frac{1}{n}\oneotimesT\nabla {\bf f}(\Prm^t) \right\|^2 } \\
        \le &-\frac{\alpha}{2} \expec{ \left\|\nabla f(\avgprm^{t}) \right\|^2 }+ \frac{\alpha L^2}{2n} \expec{ \left\| \Prm^t \right\|_{\wtk}^2 } .
	\end{aligned}
	\end{equation}
	By Assumptions \ref{assm:lip} and \ref{assm:f_var}, the third term of \eqref{proof:lemma_desc_step1} can be bounded as 
	\begin{equation}
	    \begin{aligned}
		& \frac{L}{2} \expec{\left\|\frac{\alpha}{n} \oneotimesT\nabla {\bf f}(\Prm^t;\xi^{t+1}) \right\|^2} \\
		& \leq \frac{\alpha^2 L}{2n^2}\sum_{i=1}^n \sigma_i^2 + \frac{\alpha^2 L^3}{n}  \mathbb{E}\left[ \left\| \Prm^t \right\|_{\wtk}^2\right] + \alpha^2 L \mathbb{E}\left[\left\| \nabla f(\avgprm^t) \right\|^2\right] .
	\end{aligned}
	\end{equation}
 Combining the above and setting the step size $\alpha \leq \frac{1}{4L}$ concludes the proof of the lemma.
	\hfill $\square$\\
	
\section{Proof of Lemma \ref{lemma:consensus}} \label{app:lemma_consensus_proof}
Define the following quantities:
\begin{align}
		\serr^t &:= \alpha \left(\nabla {\bf f}(\Prm^t) -\nabla {\bf f}(\Prm^t;\xi^{t+1}) \right) \notag \\
        & \quad + \alpha \theta \left( \wta^\top \rxii \wta - \wta^\top \wta(\xi_a^{t+1}) \right)\Prm^t \label{eq:stoch_err} , \\
		\gerr^t &:= \alpha \left( \nabla {\bf f}(\oneotimes \avgprm^t) - \nabla {\bf f}(\Prm^t) \right) \label{eq:global_g_err} ,\\
		{\bf e}^t_{h} &:= \alpha\theta   \wta^\top \wta(\xi_a^{t+1})  (\Prm^t - \hat{\Prm}^t), \label{eq:xhat_err} \\
		\mathbf{v}^t &:= \alpha \widetilde{\dprm}^{t} + \alpha \nabla {\bf f}(\oneotimes \avgprm^{t}) \label{eq:v_def}.
\end{align}
We notice that
	\begin{equation}
	    \begin{aligned}
		\Prm^{t+1} &= \Prm^t- \alpha (\nabla {\bf f}(\Prm^t;\xi^{t+1})+\widetilde{\lambda}^t +\theta \wta^\top \wta(\xi_a^{t+1})\hat{\Prm}^t )  \\
		&= \left({\bf I} - \alpha\theta\wta^\top \rxii \wta \right) \Prm^t - {\bf v}^t+ \serr^t + \gerr^t  +{\bf e}^t_{h}  .
		\label{proof:x_update}
	\end{aligned}
	\end{equation}
With $\expec{{\bf e}_s^t | \Prm^t} = {\bf 0}$, the consensus error can be measured as
	\begin{equation} \label{proof:lem1_step1}
	    \begin{aligned}
		& \expec{ \| \Prm^{t+1} \|_{\wtk}^2} \\
        & = \expec{\left\| \left({\bf I} - \alpha\theta\wta^\top \rxii \wta \right) \Prm^t - {\bf v}^t+ \serr^t + \gerr^t  +{\bf e}^t_{h}  \right\|_{\wtk}^2} \\
		&\le \expec{\left\| \left({\bf I} - \alpha\theta \wta^\top \rxii \wta \right) \Prm^t \right\|_{\wtk}^2 } +\expec{ \|\serr^t\|_{\wtk}^2  } \\
        &\quad + 3 \expec{ \| \gerr^t\|_{\wtk}^2  }+ 3 \expec{ \|{\bf e}^t_{h} \|_{\wtk}^2} +  3 \expec{ \|{\bf v}^t\|_{\wtk}^2  }    \\
		&\quad- 2\expec{ \dotp{\left({\bf I} - \alpha\theta \wta^\top \rxii \wta \right) \Prm^t}{ {\bf v}^t }_{ \wtk } } \\
        &\quad+ \frac{\alpha\theta \rrhomin}{2} \expec{ \dotp{\left({\bf I} - \alpha\theta \wta^\top \rxii \wta \right) \Prm^t}{ \frac{4}{\alpha\theta \rrhomin} (\gerr^t+{\bf e}^t_{h}) }_{ \wtk } } \\
		&\le (1+\frac{1}{4}\alpha\theta \rrhomin)(1-2\alpha\theta \rrhomin +\alpha^2\theta^2 \rrhomax^2)  \expec{\|\Prm^t\|^2_{\wtk}} \\
        &\quad - 2\expec{ \dotp{ \Prm^t}{ {\bf v}^t }_{ \left({\bf I} - \alpha\theta \wta^\top \rxii \wta \right)\wtk } } +(3+\frac{8}{\alpha\theta \rrhomin }) \expec{ \| \gerr^t\|_{\wtk}^2}  \\
		& \quad + (3+\frac{8}{\alpha\theta \rrhomin }) \expec{ \|{\bf e}^t_{h} \|_{\wtk}^2} + 3 \expec{ \|{\bf v}^t\|_{\wtk}^2  }+  \expec{ \|\serr^t\|_{\wtk}^2} .
	\end{aligned}
	\end{equation}
	The error quantities $\expec{ \|\serr^t\|_{\wtk}^2},\expec{\|\gerr^t\|_{\wtk}^2},\expec{ \|{\bf e}^t_{h} \|_{\wtk}^2}$ can be simplified as follows.
 Using the fact that each difference term in $\serr^t$ has mean zero and are independent when conditioned on $\Prm^t$, we obtain
 \begin{align} 
    &\expec{ \| \serr^t\|^2_{\wtk} }\le  \expec{ \| \serr^t\|^2} \notag \\
    &= \expec{\| \alpha\left(\nabla {\bf f}(\Prm^t) -\nabla {\bf f}(\Prm^t;\xi^{t+1}) \right) \|^2} \notag \\
    &\quad + \expec{ \| \alpha \theta \left( \wta^\top \rxii \wta - \wta^\top \wta(\xi_a^{t+1}) \right)\Prm^t \|^2} \notag  \\
    &\stackrel{\eqref{eq:f_var}, \eqref{eq:graph_var}}{\leq}  \alpha^2 \sum_{i=1}^n \sigma_i^2 + \alpha^2 \theta^2 \sigma_A^2 \rrhomax \expec{ \| \Prm^t \|_{\wtk}^2}, \label{eq:es_bound} 
\end{align} 
and
\begin{align} 
    &\expec{\| \gerr^t \|^2_{\wtk}} \leq \expec{ \| \gerr^t \|^2} \notag \\
    & = \alpha^2\expec{ \|\nabla {\bf f}(\oneotimes \avgprm^t) - \nabla {\bf f}(\Prm^t) \|^2 }\notag \\
    & \stackrel{\eqref{eq:f_lip}}{\leq} \alpha^2 L^2\expec{ \|\oneotimes \avgprm^t - \Prm^t \|^2} = \alpha^2 L^2 \expec{\|\Prm^t \|^2_{ \wtk }}, \label{eq:eg_bound} 
\end{align} 
and
\begin{align} 
    & \expec{ \|{\bf e}^t_{h} \|_{\wtk}^2} = \expec{\left\| \alpha\theta  \wta^\top \wta(\xi_a^{t+1}) (\hat{\Prm}^t-\Prm^t ) \right\|_{\wtk}^2} \notag \\
    &\le \alpha^2 \theta^2 \rhomax^2  \expec{ \| \hat{\Prm}^t-\Prm^t \|_{\wtk}^2} . \label{eq:eh_bound}
 \end{align}
    Now we simplify the coefficients of the consensus error term by the following step size conditions:
\begin{align}
    \begin{cases}
        \alpha^2\theta^2 \rrhomax^2 \le \frac{1}{16}\alpha\theta \rrhomin & \Leftrightarrow \alpha \leq \frac{\rrhomin}{16 \rrhomax^2 \theta}, \\[0.5em]
        \alpha^2\theta^2\sigma_A^2 \rrhomax \le \frac{1}{16}\alpha\theta \rrhomin & \Leftrightarrow \alpha \leq \frac{\rrhomin}{16 \sigma_A^2 \rrhomax \theta} , \\[0.5em] 
        3 \leq \frac{1}{\alpha \theta \rrhomin} & \Leftrightarrow \alpha \leq \frac{1}{3 \rrhomin \theta}, \\[0.5em]
        \frac{9\alpha L^2}{\theta \rrhomin} \le \frac{1}{16}\alpha\theta \rrhomin & \Leftrightarrow \theta \ge \frac{12 L}{\rrhomin}. 
    \end{cases}
\end{align}
Combining with the upper bounds of $\expec{ \|\serr^t\|_{\wtk}^2}$, $\expec{\|\gerr^t\|_{\wtk}^2}$, $\expec{ \|{\bf e}^t_{h} \|_{\wtk}^2}$, the proof is completed.
	\hfill $\square$ \\

	\section{Proof of Lemma \ref{lemma:dual_err}} \label{app:lemma_dualerr_proof}
    Consider expanding ${\bf v}^{t+1}$ as
	\begin{equation}
	    \begin{aligned}
		&{\bf v}^{t+1} \stackrel{\eqref{eq:v_def}}{=} \alpha \widetilde{\lambda}^{t+1} + \alpha \nabla {\bf f}(\oneotimes \avgprm^{t+1}) \\
		&={\bf v}^t + \alpha\eta \wta^\top \wta(\xi_a^{t+1}) \Prm^t + \alpha\eta \wta^\top \wta(\xi_a^{t+1}) (\hat{\Prm}^t -\Prm^t) \\
        &+ \alpha \nabla {\bf f}(\oneotimes \avgprm^{t+1}) -  \alpha \nabla {\bf f}(\oneotimes \avgprm^{t}) .
		\label{proof:v_update}
	\end{aligned}
	\end{equation}
    To simplify notations, denote $\mathcal{N}_v^t := \alpha\eta \wta^\top \wta(\xi_a^{t+1}) \Prm^t + \alpha\eta \wta^\top \wta(\xi_a^{t+1}) (\hat{\Prm}^t -\Prm^t) + \alpha \nabla {\bf f}(\oneotimes \avgprm^{t+1}) -  \alpha \nabla {\bf f}(\oneotimes \avgprm^{t}) $. 
    By Assumption~\ref{assm:rand-graph}, for any ${\bf y} \in \mathbb{R}^{nd}$ it holds that
    \begin{equation}
        \expec{ \left\| {\bf y} \right\|_{\wtq + \c\wtk}^2 } \leq (\rrhomin^{-1} + \c) \expec{ \left\| {\bf y} \right\|_{\wtk}^2 }.
    \end{equation}
	Therefore,
	\begin{equation}
	    \begin{aligned}
		&\expec{ \left\| {\bf v}^{t+1} \right\|_{\wtq + \c {\wtk}}^2} \leq \expec{\left\| {\bf v}^t \right\|_{\wtq + \c \wtk}^2} 
        + 2\expec{ \dotp{{\bf v}^t }{\mathcal{N}_v^t}_{\wtq + \c\wtk}} \\
		&+ 4\alpha^2\eta^2  \rhomax^2 (\rrhomin^{-1} + \c) \left(\expec{ \| \Prm^t \|_{\wtk}^2 } +   \expec{ \| \hat{\Prm}^t -\Prm^t \|_{\wtk}^2 }\right)\\
        &+ 2\alpha^2(\rrhomin^{-1} + \c)\expec{ \left\|   \nabla {\bf f}(\oneotimes \avgprm^{t+1}) -  \nabla {\bf f}(\oneotimes \avgprm^{t})  \right\|^2} .  \label{proof:lem2_step1}
	\end{aligned}
	\end{equation}
    
	The second term of \eqref{proof:lem2_step1} can be simplified as
	\begin{equation}
	    \begin{aligned}
		&2\expec{\dotp{{\bf v}^t }{ \mathcal{N}_v^t}_{\wtq + \c\wtk}} \leq 2\alpha\eta \expec{\dotp{ {\bf v}^t }{ \Prm^t}_{(\wtq + 
				\c\wtk)\wta^\top   \rxi \wta}} \\
		& + 2 \alpha \expec{ \left\| {\bf v}^t \right\|_{\wtq + \c\wtk}^2 }+\alpha\eta^2 \rhomax^2 \expec{\left\| \hat{\Prm}^t -\Prm^t \right\|_{\wtq + \c\wtk}^2}    \\
		&+ \alpha (\rrhomin^{-1} + \c) \expec{\left\|  \nabla {\bf f}(\oneotimes \avgprm^{t+1}) -   \nabla {\bf f}(\oneotimes \avgprm^{t})  \right\|^2} .
		\label{proof:lem2_step3}
	\end{aligned}
	\end{equation}
	The proof is concluded by combining the above inequalities with the auxiliary Lemma \ref{lemma:update_err} (to be shown later) and simplifying $\wtq \wta^\top   \rxi \wta = \wtk$ and $\wtk \wta^\top   \rxi \wta = \wta^\top   \rxi \wta$.
	\hfill $\square$ \\

\section{Proof of Lemma \ref{lemma:xv_inner}}   \label{app:lemma_xvinner_proof}
We denote $\mathcal{N}_{\rxii}={\bf I}-\alpha\theta\wta^\top \rxii\wta $.
By \eqref{proof:x_update} and \eqref{proof:v_update}, it holds
\begin{equation}
	\begin{aligned}
		&\expec{\dotp{\Prm^{t+1}}{\mathbf{v}^{t+1}}_{\wtk}} = \expec{\dotp{\Prm^t}{\mathbf{v}^t}_{\mathcal{N}_{\rxii}\wtk - \alpha\eta  \wtk \wta^\top   \rxi \wta  }} \\
		&+\alpha \eta \expec{\|\Prm^{t}\|^2_{ \mathcal{N}_{\rxii}\wtk \wta^\top \rxii  \wta}}- \expec{\| \mathbf{v}^t \|^2_{\wtk}}  \\
		& +\alpha \expec{\dotp{\Prm^t}{\nabla {\bf f}(\oneotimes \avgprm^{t+1}) - \nabla {\bf f}(\oneotimes \avgprm^{t})}_{\mathcal{N}_{\rxii}\wtk}} \\
		&
		-\alpha \expec{\dotp{\mathbf{v}^{t}}{\nabla {\bf f}(\oneotimes \avgprm^{t+1}) - \nabla {\bf f}(\oneotimes \avgprm^{t}) }_{\wtk}} \\
		& -\expec{\dotp{\mathcal{N}_{\rxii} \Prm^t - \mathbf{v}^{t}}{\frac{\eta}{\theta}{\bf e}_h^t}_{\wtk}} \\
		& +\expec{\dotp{ \gerr^t+{\bf e}^t_{h} }{\mathbf{v}^{t} + \alpha\eta \wta^\top  \rxii \wta \Prm^t - \frac{\eta}{\theta} {\bf e}^t_h }_{\wtk}} \\
		& +\alpha\expec{\dotp{ {\bf e}^t_s + \gerr^t+{\bf e}^t_{h} }{\nabla {\bf f}(\oneotimes \avgprm^{t+1}) - \nabla {\bf f}(\oneotimes \avgprm^{t})}_{\wtk}} .
	\end{aligned}
\end{equation}
Now notice that by applying Young's inequality on the fourth to eighth terms of the above, we get
\begin{equation} \notag
	\begin{aligned}
		&\alpha \expec{\dotp{\Prm^t}{\nabla {\bf f}(\oneotimes \avgprm^{t+1}) - \nabla {\bf f}(\oneotimes \avgprm^{t})}_{\mathcal{N}_{\rxii}\wtk}} \\
		&\leq \frac{\alpha}{2} \expec{ \| ({\bf I} - \alpha \theta \wta^\top \rxii \wta ) \Prm^t \|^2_{\wtk}} \\
		&\quad + \frac{\alpha}{2} \expec{\|\nabla {\bf f}(\oneotimes \avgprm^{t+1}) - \nabla {\bf f}(\oneotimes \avgprm^{t}) \|^2 },
	\end{aligned}
\end{equation}
and
\begin{equation} \notag
	\begin{aligned}
		&-\alpha \expec{\dotp{\mathbf{v}^{t}}{\nabla {\bf f}(\oneotimes \avgprm^{t+1}) - \nabla {\bf f}(\oneotimes \avgprm^{t}) }_{\wtk}} \\
		&\leq \frac{\alpha}{2} \expec{ \| {\bf v}^t \|^2_{\wtk}} + \frac{\alpha}{2} \expec{\|\nabla {\bf f}(\oneotimes \avgprm^{t+1}) - \nabla {\bf f}(\oneotimes \avgprm^{t}) \|^2 } ,
	\end{aligned}
\end{equation}
and
\begin{equation} \notag
	\begin{aligned}
		&-\alpha \expec{\dotp{\mathcal{N}_{\rxii}\Prm^t - \mathbf{v}^{t}}{\eta \wta^\top \wta(\xi_a^{t+1}) (\hat{\Prm}^t-\Prm^t)}_{\wtk}} \\
		& \le \alpha \expec{\| \left({\bf I} - \alpha\theta\wta^\top \rxii \wta \right) \Prm^t \|^2_{\wtk}} + \alpha \expec{\|  \mathbf{v}^{t} \|^2_{\wtk}} \\
		& \quad + \frac{\alpha\eta^2 \rrhomax^2}{2} \expec{\| \hat{\Prm}^t-\Prm^t \|_{\wtk}^2}  ,
	\end{aligned}
\end{equation}
and
\begin{equation} \notag
	\begin{aligned}
		&\expec{\dotp{ \gerr^t+{\bf e}^t_{h} }{\mathbf{v}^{t} + \alpha\eta \wta^\top  \rxii \wta \Prm^t - \frac{\eta}{\theta} {\bf e}_h^t }_{\wtk}} \\
		& \quad +\alpha\expec{\dotp{ {\bf e}_s^t + \gerr^t+{\bf e}^t_{h} }{\nabla {\bf f}(\oneotimes \avgprm^{t+1}) - \nabla {\bf f}(\oneotimes \avgprm^{t})}_{\wtk}} \\
		&\leq \frac{3 \alpha}{2} \expec{\| {\bf e}_s^t \|^2_{\wtk}} + (1 + \frac{1}{2} + \frac{3 \alpha}{2}) \expec{\| \gerr^t \|^2_{\wtk}} \\
		&\quad +  (1 + \frac{\eta^2}{2 \theta^2} - \frac{\eta}{\theta} + \frac{3 \alpha}{2}) \expec{\| {\bf e}_h^t \|^2_{\wtk}} \\
		& \quad + \frac{1}{2}(1 + \frac{1}{2}) \expec{\|\mathbf{v}^{t}\|^2_{\wtk}} + \frac{1}{2}(1 + 2) \alpha^2\eta^2 \expec{\|\wta^\top   \rxi \wta \Prm^t \|^2_{\wtk}}\\ 
		& \quad +\frac{\alpha}{2} \expec{\| \nabla {\bf f}(\oneotimes \avgprm^{t+1}) - \nabla {\bf f}(\oneotimes \avgprm^{t}) \|^2_{\wtk}} .
	\end{aligned}
\end{equation}
By the step size conditions $\alpha \leq 1$ and $\eta \leq \theta$, we can combine the above inequalities to get
	\begin{align*}
			&\expec{\dotp{\Prm^{t+1}}{\mathbf{v}^{t+1}}_{\wtk}} \\
			&\leq \expec{\dotp{\Prm^t}{\mathbf{v}^t}_{\wtk-(\alpha\theta+\alpha\eta)\wta^\top \rxii \wta }} + ( \frac{3 \alpha}{2} - \frac{1}{4}) \expec{ \| {\bf v}^t \|^2_{\wtk}}  \\
			&\quad +  \frac{3 \alpha}{2}( 1 - 2 \alpha \theta \rrhomin +\alpha^2\theta^2 \rrhomax^2 + \alpha \eta^2 \rrhomax^2) \expec{\|\Prm^t\|^2_{\wtk}}   \\
			&\quad + \frac{\alpha\eta^2 \rrhomax^2 }{2} \expec{ \| \hat{\Prm}^t-\Prm^t \|_{\wtk}^2} \\
			&\quad + \frac{3 \alpha}{2} \expec{\| \nabla {\bf f}(\oneotimes \avgprm^{t+1}) - \nabla {\bf f}(\oneotimes \avgprm^{t}) \|^2_{\wtk}} \\
			&\quad + \frac{3\alpha}{2} \expec{\| {\bf e}_s^t \|^2_{\wtk}} + 3 \expec{\| \gerr^t \|^2_{\wtk}} + \frac{5}{2} \expec{\| {\bf e}_h^t \|^2_{\wtk}} .
	\end{align*}
Now apply the inequalities in \eqref{eq:es_bound}, \eqref{eq:eg_bound}, \eqref{eq:eh_bound} to see that
\begin{equation}
	\begin{aligned}
		\begin{aligned}
			&\expec{\dotp{\Prm^{t+1}}{\mathbf{v}^{t+1}}_{\wtk}} \\
			&\leq \expec{\dotp{\Prm^t}{\mathbf{v}^t}_{\wtk-(\alpha\theta+\alpha\eta)\wta^\top \rxii \wta }} + ( \frac{3 \alpha}{2} - \frac{1}{4}) \expec{ \| {\bf v}^t \|^2_{\wtk}}  \\
			&\quad +  \frac{3 \alpha}{2}(1-2\alpha\theta \rrhomin + \alpha^2\theta^2\rrhomax^2 + \alpha \eta^2 \rrhomax^2 \\
			&\qquad \qquad + \alpha^2 \theta^2 \sigma_A^2 \rrhomax + 2 \alpha L^2) \expec{\|\Prm^t\|^2_{\wtk}}   \\
			&\quad + \frac{1 }{2} ( \alpha\eta^2 \rrhomax^2 + 5 \alpha^2 \theta^2 \rhomax^2 ) \expec{ \| \hat{\Prm}^t-\Prm^t \|_{\wtk}^2} \\
			&\quad + \frac{3 \alpha}{2} \expec{\| \nabla {\bf f}(\oneotimes \avgprm^{t+1}) - \nabla {\bf f}(\oneotimes \avgprm^{t}) \|^2_{\wtk}} + \frac{3\alpha^3}{2} \sum_{i=1}^n \sigma_i^2 .
		\end{aligned}
	\end{aligned}
\end{equation}
Then, under the following step size conditions:
\begin{equation}
	\begin{cases}
		\alpha^2\theta^2\rrhomax^2 \leq \alpha \theta \rrhomin / 2 & \Leftrightarrow \alpha \leq \frac{\rrhomin}{2 \rrhomax^2 \theta }, \\[0.5em]
		\alpha \eta^2 \rrhomax^2  \leq \alpha \theta \rrhomin / 2 & \Leftrightarrow \theta \ge  \frac{2\eta^2\rrhomax^2}{\rrhomin}
         , \\[0.5em]
		\alpha^2 \theta^2 \sigma_A^2 \rrhomax  \leq \alpha \theta \rrhomin / 2 & \Leftrightarrow \alpha \leq \frac{\rrhomin}{2 \sigma_A^2 \rrhomax \theta}, \\[0.5em]
		2 \alpha L^2 \leq \alpha \theta \rrhomin / 2 & \Leftrightarrow \theta \ge 4 L^2 / \rrhomin,
	\end{cases}
\end{equation}
and $\alpha \leq 1/12$, the proof is completed by applying the result of Lemma \ref{lemma:update_err}.
\hfill $\square$
 
	
\section{Proof of Lemma \ref{lemma:xhat-err}}  \label{app:lemma_xhaterr_proof}
By Assumption \ref{assm:compress}, it holds that
	\begin{equation}
		\begin{aligned}
			&\expec{\|\hat{\Prm}^{t+1}-\Prm^{t+1}\|^2} \\
            &= \expec{ \| \hat{\Prm}^t - \Prm^{t+1} + \gamma Q(\Prm^{t+1}-\hat{\Prm}^t;\xi_q^{t+1}) \|^2 }\\
			& \le (1-\update{r}\gamma\delta) \expec{\|\Prm^{t+1}-\Prm^t + \Prm^t - \hat{\Prm}^t\|^2}+ \gamma^2\sigma_{\xi}^2 \\
			& \le (1-\update{r}\gamma\delta)(1+\frac{1}{\tau})\expec{\|\Prm^t - \hat{\Prm}^t\|^2} \\
            & \quad + (1-\update{r}\gamma\delta)(1+\tau)\expec{\|\Prm^{t+1}-\Prm^t \|^2} + \gamma^2\sigma_{\xi}^2 \\
			&\leq (1-\frac{\update{r}\gamma\delta}{2})
			 \expec{ \|\Prm^t - \hat{\Prm}^t\|^2} + (\frac{2}{\update{r}\gamma\delta} - 2)\expec{\|\Prm^{t+1}-\Prm^t \|^2}  \\
            &\quad+ \gamma^2\sigma_{\xi}^2 ,
		\end{aligned}
	\end{equation}
	where the last inequality is obtained by choosing $\tau = \frac{1- \update{r}\gamma \delta}{\update{r}\gamma \delta / 2}$. By \eqref{proof:x_update}, we get
	\begin{align*}
		&\expec{\|\Prm^{t+1}-\Prm^t \|^2}  \\
		&\le  2\alpha^2 \theta^2 \expec{\left\| \wta^\top \rxii \wta \Prm^t \right\|^2 } + 2\alpha \theta \expec{ \dotp{ \Prm^t}{ {\bf v}^t }_{ \wta^\top \rxii \wta \wtk } }  \\
		& \quad + 2 \expec{\|{\bf v}^t\|^2}+ 2 \expec{ \| \gerr^t\|^2  } + 2 \expec{ \|{\bf e}^t_{h} \|^2} +  \expec{ \|\serr^t\|^2  } \\
		&\leq \left( 2 \alpha^2\theta^2 \rrhomax^2 + 2 \alpha^2 L^2+ \alpha^2 \theta^2 \sigma_A^2  \rrhomax  \right) \expec{\left\| \Prm^t \right\|^2_{\wtk} } \\
        & \quad + 2\alpha \theta \expec{ \dotp{ \Prm^t}{ {\bf v}^t }_{\wta^\top \rxii \wta } } + 2 \expec{\|{\bf v}^t\|^2} \\
		& \quad + 2 \alpha^2 \theta^2 \rhomax^2\expec{ \| \hat{\Prm}^t-\Prm^t \|^2}+ \alpha^2 \sum_{i=1}^n \sigma_i^2,
	\end{align*}
	where the last inequality is due to \eqref{eq:es_bound}, \eqref{eq:eg_bound}, \eqref{eq:eh_bound}.
	By the fact that $\widetilde{\lambda}^t = \wta^\top \lambda^t \Rightarrow ({\bf 1} \otimes {\bf I_d})^\top \widetilde{\lambda}^t = 0$, we have
	\begin{equation}
	    \begin{aligned}
		&\|{\bf v}^t\|_{\wtk}^2 = \|\alpha \widetilde{\lambda}^{t} + \alpha \nabla {\bf f}(\oneotimes \avgprm^{t})\|_{({\bf I} - \frac{1}{n} {\bf 1}{\bf 1}^\top) \otimes {\bf I}_d}^2 \\
		&= \|{\bf v}^t\|^2 - \alpha^2 \| \nabla {\bf f}(\oneotimes \avgprm^{t}) \|^2_{\frac{1}{n} {\bf 1}{\bf 1}^\top \otimes {\bf I}_d} \\
		&= \|{\bf v}^t\|^2 - \alpha^2 \left\| \nabla f(\avgprm^{t}) \right\|^2 .
	\end{aligned}
	\end{equation}
	Combining the above and applying $2\alpha^2 \theta^2 \rhomax^2 \leq \gamma \delta / 4 \Leftrightarrow \alpha \leq \sqrt{\gamma \delta / (8 \rhomax^2 \theta^2 )}$ completes the proof.
	\hfill $\square$

\section{Proof of Lemma \ref{lemma:para_choose}}   \label{app:lemma_parachoose_proof}
We first obtain a lower bound for $F_t$. By the inequality $|\dotp{ {\bf x} }{ {\bf y} } | \leq \frac{1}{2\delta_0} \| {\bf x} \|^2 + \frac{\delta_0}{2} \| {\bf y} \|^2$ for any $\delta_0> 0$,
	\begin{equation}
	    \begin{aligned}
		&F_t \ge f(\avgprm^{t})  +  \a\| \Prm^t \|_{\wtk}^2 +  \| {\bf v}^t \|_{\b\wtq + \b\c \wtk}^2 - \frac{\d}{2 \delta_0} \| \Prm^t \|^2_{\wtk} \\
        &\quad- \frac{\d \delta_0}{2} \| {\bf v}^t\|^2_{\wtk}  +\e \| \hat{\Prm}^t-\Prm^t \|^2 \\
		&= f(\avgprm^{t})  +  (\a - \frac{\d}{2 \delta_0}) \| \Prm^t \|_{\wtk}^2 +   \| {\bf v}^t \|_{\b \wtq + (\b\c - \d \delta_0 /2) \wtk}^2 \\
        &\quad +\e\| \hat{\Prm}^t-\Prm^t \|^2 \\
		&\stackrel{(\delta_0 = \frac{\d}{2\a} )}{=} f(\avgprm^{t})   +   \| {\bf v}^t \|_{\b \wtq + (\b\c - \frac{\d^2}{4\a}) \wtk}^2 +\e \| \hat{\Prm}^t-\Prm^t \|^2 \\
		&\stackrel{\eqref{eq:q_ineq}}{\ge} f(\avgprm^{t}) + \| {\bf v}^t \|_{(\b \cdot \rhomax^{-1} + \b\c - \frac{\d^2}{4\a}) \wtk}^2+\e \| \hat{\Prm}^t-\Prm^t \|^2 . \label{eq:potential_lb}
	\end{aligned}
	\end{equation}
    To simplify notation, we set $\delta_{1} = \frac{1024\rhomax}{\gamma\delta\rrhomin}, \delta_2 = 8\rhomax \rrhomin^{-1}$ in this section, i.e. $\d = \delta_1 \a,  \e = \delta_2\a$.
	By choosing $\delta_1\ge 2\delta_2+2, \alpha\le \frac{4}{\eta\rhomax\delta_1^2}$, it holds 
	\begin{equation}
		\begin{aligned}
		    &\b \cdot \rhomax^{-1} + \b\c - \frac{\d^2}{4\a} \\
            &\quad = \a \left( \frac{1}{\alpha\eta\rhomax}+\alpha (\frac{\delta_1}{2}(\theta+\eta)-\theta)-\delta_2 \alpha\theta-\frac{\delta_1^2}{4} \right) \ge 0.
		\end{aligned}
	\end{equation}
	Therefore, $F_t \ge f(\avgprm^{t}) \ge f^* > -\infty$.
	If the parameters are chosen according to \eqref{eq:abcd_choice}, then 
    \begin{equation}
        \begin{aligned}
            \wxv &= -\a \cdot 2 \wtk + \a \cdot 2\alpha\theta \wta^\top \rxii \wta + \b \cdot 2\alpha\eta (\wtk+\c \wta^\top \rxi \wta) \\ 
            &\quad - \d \cdot (\alpha\theta+\alpha\eta) \wta^\top \rxii \wta +  \e \frac{4\alpha\theta}{\update{r}\gamma\delta}\wta^\top \rxii \wta =0 ,
        \end{aligned}
    \end{equation}
    and the inner product term in \eqref{eq:Ft} vanishes. 
    Consider the iterative relationship of the potential function. Combining the previous lemma, we can obtain 
    \begin{equation} \label{eq:Ft_decrease}
    \begin{aligned}
         \expec{F_{t+1}} &\leq \expec{F_t} + \bar{\omega}_{f}\expec{\left\| \nabla f(\avgprm^{t}) \right\|^2}+\alpha^2\bar{\omega}_{\sigma} \bar{\sigma}^2 \\
        &~ + \bar{\omega}_x\expec{\| \Prm^t \|^2_{ \wtk}}+ \bar{\omega}_v \expec{\| {\bf v}^t \|^2_{\wtk}} +8\a \gamma^2\sigma_{\xi}^2\update{\frac{\rhomax}{\rrhomin}}\\
        &~ + \bar{\omega}_{\hat{x}} \expec{\| \hat{\Prm}^t-\Prm^t \|^2} + \expec{\dotp{\Prm^t}{{\bf v}^t }_{ \wxv }} ,
    \end{aligned}
    \end{equation}
where the coefficients are given by
\begin{equation}
    \begin{aligned}
        \bar{\omega}_{f} &= -\frac{\alpha}{4} +\b\cdot 6\alpha^3(\rrhomin^{-1}+\c)n L^2 + \d\cdot \frac{9}{2} \alpha^3 n L^2 \\
        &\quad + \e\cdot  \frac{4}{\update{r}\gamma\delta}\alpha^2 ,
    \end{aligned}
\end{equation}
and
\begin{equation}
    \begin{aligned}
    \bar{\omega}_x &= \frac{\alpha L^2}{n}  - \a\cdot \frac{3}{2}\alpha\theta\rrhomin \\
    &\quad + \b\cdot(4\alpha^2\eta^2\rhomax^2 + 6\alpha^3 L^4)(\rrhomin^{-1}+\c) +\d\cdot \frac{3\alpha}{2}\\
    &\quad + \e\cdot \frac{4}{\update{r}\gamma\delta}[\alpha^2\theta^2(\rrhomax^2+\sigma_{A}^2\rrhomax/2)+\alpha^2L^2],
    \end{aligned}
\end{equation}
and
\begin{equation}
    \bar{\omega}_v =  \a\cdot 3 + \b\cdot 2\alpha (\rrhomin^{-1} + \c) - \d\cdot \frac{1}{8} +\e\cdot \frac{4}{\update{r}\gamma\delta},
\end{equation}
and
\begin{equation}
    \begin{aligned} \textstyle
        \bar{\omega}_{\hat{x}} &= \a\cdot 9 \alpha \theta \rhomax^2 \rrhomin^{-1}
        + \b\cdot 2\alpha\eta^2 \rhomax^2 (\rrhomin^{-1} + \c)   \\
        & \quad + \d\cdot \frac{1}{2}( \alpha\eta^2 \rrhomax^2 + 5 \alpha^2 \theta^2 \rhomax^2 )
        +\e\cdot-\frac{\update{r}\gamma\delta}{4} ,
    \end{aligned}
\end{equation}
and
\begin{equation}
    \begin{aligned}
        \bar{\omega}_{\sigma} &=  \frac{L}{2n} + \a n+ \b\cdot 6\alpha(\rrhomin^{-1}+\c) L^2 + \d\cdot (\frac{9\alpha L^2}{2}+\frac{3\alpha n}{2})\\
        &\quad + \e\cdot  \frac{2n}{\update{r}\gamma\delta} .
    \end{aligned}
\end{equation}

By choosing $\alpha \le \rrhomin^{-1} (\frac{\delta_1}{2}(\theta+\eta)-\theta)^{-1}$ and \eqref{eq:q_ineq}, we observe that \begin{equation}
	    \c \le \rrhomin^{-1}.
	\end{equation}
 To upper bound $\bar{\omega}_x$.
 \begin{equation}
    \begin{aligned}
    &\bar{\omega}_x \le \frac{\alpha L^2}{n}  - \a\cdot \frac{3}{2}\alpha\theta\rrhomin \\
    &+ \a\cdot(8\alpha\eta\rhomax^2 + 12\alpha^2\eta^{-1} L^4)\rrhomin^{-1} +\a\cdot \frac{3\delta_1\alpha}{2}\\
    &+\a\cdot \frac{4\delta_2}{\update{r}\gamma\delta}[\alpha^2\theta^2(\rrhomax^2+\sigma_{A}^2\rrhomax/2)+\alpha^2L^2] .
    \end{aligned}
\end{equation}
It holds that $\omega_x  \le \frac{1}{4} \alpha\theta\rrhomin \a =: \omega_x< 0$, with
\begin{align}
    \begin{cases}
        \frac{\alpha L^2}{n} \le \frac{1}{8} \alpha\theta\rrhomin \a &\Leftrightarrow  \theta \ge \frac{8L^2}{\rrhomin n\a} , \\
        8\alpha\eta\rhomax^2\rrhomin^{-1}  \le \frac{1}{4} \alpha\theta\rrhomin  &\Leftrightarrow \theta \ge 32\eta\rhomax^2\rrhomin^{-2}  , \\
        12\alpha^2\eta^{-1} L^4\rrhomin^{-1}    \le \frac{1}{8} \alpha\theta\rrhomin  &\Leftrightarrow \alpha \le \frac{\theta\eta\rrhomin^2}{96L^4}  , \\
        \frac{3}{2}\delta_1\alpha   \le \frac{3}{8} \alpha\theta\rrhomin  &\Leftrightarrow \theta \ge 8\delta_1\rrhomin^{-1}  , \\
        \frac{4\delta_2}{\update{r}\gamma\delta}\alpha^2\theta^2\rrhomax^2    \le \frac{1}{8} \alpha\theta\rrhomin  &\Leftrightarrow  \alpha\le \frac{\update{r}\gamma\delta\rrhomin}{32\delta_2\theta\rrhomax^2} , \\
        \frac{2\delta_2}{\update{r}\gamma\delta}\alpha^2\theta^2\sigma_{A}^2\rrhomax \le \frac{1}{8} \alpha\theta\rrhomin  &\Leftrightarrow \alpha\le\frac{\update{r}\gamma\delta\rrhomin}{16\delta_2\theta\sigma_{A}^2\rrhomax}  , \\
        \frac{4\delta_2}{\update{r}\gamma\delta}\alpha^2L^2    \le \frac{1}{8} \alpha\theta\rrhomin  &\Leftrightarrow \alpha\le\frac{\update{r}\gamma\delta\theta\rrhomin}{32\delta_2L^2}  .
    \end{cases}
\end{align}
To upper bound $\bar{\omega}_v$, using $\c \leq \rrhomin^{-1}$ leads to
	\begin{align}
		\bar{\omega}_v \le & \a\cdot 3 + \a\cdot 4\rrhomin^{-1} \eta^{-1}+ \a\cdot -\frac{1}{8}\delta_1 +\a\cdot \frac{4}{\update{r}\gamma\delta} \delta_2 .
	\end{align}
	It holds that $\omega_v  \le -\a\frac{1}{32}\delta_1 \le -\a =: \omega_v< 0$, with 
	\begin{align}
		\begin{cases}
			3 \le \frac{1}{32}\delta_1 &\Leftrightarrow  \delta_1\ge 96 , \\
			4\rrhomin^{-1} \eta^{-1} \le \frac{1}{32}\delta_1 &\Leftrightarrow  \delta_1\ge 128\rrhomin^{-1} \eta^{-1},\\
			\frac{4}{\update{r}\gamma\delta} \delta_2 \le \frac{1}{32}\delta_1 &\Leftrightarrow \delta_1 \ge \frac{128}{\update{r}\gamma\delta} \delta_2.
		\end{cases}
	\end{align}
To upper bound $\bar{\omega}_{\hat{x}}$,
	\begin{equation}
	    \begin{aligned}
		&\omega_{\hat{x}} \le   \a\cdot 9 \alpha \theta \rhomax^2 \rrhomin^{-1} + \a\cdot 4\eta \rhomax^2 \rrhomin^{-1}  \\
		&+ \a\cdot  \frac{\delta_1}{2}( \alpha\eta^2 \rrhomax^2 + 5 \alpha^2 \theta^2 \rhomax^2 )
        - \a\cdot \frac{\update{r}\gamma\delta}{4}\delta_2. 
	\end{aligned}
	\end{equation}
	It holds that $\bar{\omega}_{\hat{x}} \le -\frac{\update{r}\gamma\delta}{16} \a =: \omega_{\hat{x}} <0$, with
	\begin{align}
		\begin{cases}
			9 \alpha \theta \rhomax^2 \rrhomin^{-1} \le \frac{\update{r}\gamma\delta}{16}\delta_2 &\Leftrightarrow \alpha \le \frac{\update{r}\gamma\delta \delta_2 \rrhomin}{144\theta \rhomax^2}, \\
			4\eta \rhomax^2 \rrhomin^{-1} \le \frac{\update{r}\gamma\delta}{16}\delta_2  &\Leftrightarrow \delta_2 \ge \frac{64\eta \rhomax^2 }{\update{r}\gamma\delta\rrhomin} , \\
            \frac{\delta_1}{2}\alpha\eta^2 \rrhomax^2  \le \frac{\update{r}\gamma\delta}{32}\delta_2  &\Leftrightarrow \alpha \le \frac{\update{r}\gamma\delta \delta_2}{16 \delta_1 \eta^2 \rrhomax^2 } ,\\
			\frac{\delta_1}{2} 5\alpha^2 \theta^2 \rhomax^2 \le \frac{\update{r}\gamma\delta}{32}\delta_2  &\Leftrightarrow \alpha \le \frac{\sqrt{\update{r}\gamma\delta \delta_2/80\delta_1}}{ \theta \rhomax} .
		\end{cases}
	\end{align}
To upper bound $\bar{\omega}_{f}$, notice that
	\begin{align}
		\omega_{f} \le - \frac{1}{4}\alpha +12\a\alpha^2\eta^{-1} \rrhomin^{-1}nL^2+ \frac{9}{2}\a \delta_1  \alpha^2 n L^2+\a \frac{4\delta_2}{\update{r}\gamma\delta} \alpha^2 .
	\end{align}
    It holds that $\bar{\omega}_{f} \le -\frac{1}{16}\alpha =: \omega_{f} < 0$, with 
	\begin{align}
		\begin{cases}
			12\a\alpha^2\eta^{-1} \rrhomin^{-1}nL^2 \le \frac{1}{16}\alpha &\Leftrightarrow \alpha \le \frac{\eta\rrhomin}{192nL^2 \a} ,\\
			\frac{9}{2}\a \delta_1 \alpha^2nL^2 \le \frac{1}{16}\alpha &\Leftrightarrow \alpha \le \frac{1}{72\delta_1nL^2\a} ,\\
			\a \frac{4\delta_2}{\update{r}\gamma\delta} \alpha^2\le \frac{1}{16}\alpha &\Leftrightarrow \alpha \le \frac{\update{r}\gamma\delta}{64\delta_2 \a} .
		\end{cases}
	\end{align}
Lastly, the upper bound of $\bar{\omega}_{\sigma}$ is
    \begin{align}
		\bar{\omega}_{\sigma}  \le \frac{L}{2n} + \a (n +12\frac{L^2}{\eta\rrhomin}+\frac{3\delta_1\alpha}{2}(3L^2+n)+\frac{2n\delta_2}{\update{r}\gamma\delta} ) .
	\end{align}
It holds that $\bar{\omega}_{\sigma} \le \omega_{\sigma}$ under $\theta \ge \frac{8192\rhomax}{\update{r}\gamma\delta\rrhomin^2}$.
    
We notice that a sufficient condition to satisfying the above step size conditions is to set $\theta \geq \theta_{lb}, \alpha \leq \alpha_{ub}$ where
    \begin{equation}
        \begin{aligned}
            &\theta_{lb} = 4\rrhomin^{-1}\max\{\frac{2L^2}{ n\a},\frac{2048\rhomax}{\update{r}\gamma\delta\rrhomin},L^2\} ,\\
            &\alpha_{ub} = \frac{\update{r}\gamma\delta}{256\theta} \cdot\min\{\frac{\rrhomin^2}{\rrhomax^2\rhomax},\frac{\rrhomin^2}{\sigma_{A}^2\rrhomax\rhomax},  \\
            &\qquad \qquad \frac{1}{72n\a \rhomax},\frac{\rrhomin}{2\rhomax \a}\} .
        \end{aligned}
    \end{equation}
This concludes the proof. \hfill $\square$ \\

\section{Auxiliary Lemmas}
    \begin{lemma} \label{lemma:update_err}  Under Assumption \ref{assm:lip} and \ref{assm:f_var}, 
        \begin{equation} \notag
            \begin{aligned}
                &\expec{ \left\|   \nabla {\bf f}(\oneotimes \avgprm^{t+1}) - \nabla {\bf f}(\oneotimes \avgprm^{t}) \right\|^2 }\\ 
                & \le 3\alpha^2 n L^2 \left\{  \frac{1}{n^2}  \sum_{i=1}^n  \sigma_i^2  +\expec{\left\| \nabla f(\avgprm^{t}) \right\|^2}  \right\} +3\alpha^2 L^4 \|\Prm^t\|^2_{\wtk} .
            \end{aligned}
        \end{equation} 
    \end{lemma}
\begin{proof} By the Lipschitz gradient assumption on each local objective function $f_i$,
	\begin{equation} \notag
	    \begin{aligned}
		&\expec{ \left\|   \nabla {\bf f}(\oneotimes \avgprm^{t+1}) - \nabla {\bf f}(\oneotimes \avgprm^{t}) \right\|^2 }  \\
		& \le 3\alpha^2 n L^2 \expec{ \frac{1}{n^2}  \sum_{i=1}^n  \left\|\nabla f_i(\Prm_i^t; \xi_i^{t+1}) - \nabla f_i(\Prm_i^t) \right\|^2} \\
        &+\expec{\left\| \nabla f(\avgprm^{t}) \right\|^2+\frac{1}{n} \sum_{i=1}^n \left\| \nabla f_i(\Prm^t)-\nabla f_i(\avgprm^t) \right\|^2}   \\
		& \leq 3\alpha^2 n L^2 \left\{  \frac{1}{n^2}  \sum_{i=1}^n  \sigma_i^2  +\expec{\left\| \nabla f(\avgprm^{t}) \right\|^2}  \right\} +3\alpha^2 L^4 \|\Prm^t\|^2_{\wtk} .
	\end{aligned}
	\end{equation}
\end{proof}


\begin{thebibliography}{10}

\bibitem{dimakis2010gossip}
A.~G. Dimakis, S.~Kar, J.~M. Moura, M.~G. Rabbat, and A.~Scaglione, ``Gossip algorithms for distributed signal processing,'' {\em Proceedings of the IEEE}, vol.~98, no.~11, pp.~1847--1864, 2010.

\bibitem{xiao2007distributed}
L.~Xiao, S.~Boyd, and S.-J. Kim, ``Distributed average consensus with least-mean-square deviation,'' {\em Journal of parallel and distributed computing}, vol.~67, no.~1, pp.~33--46, 2007.

\bibitem{kar2013consensus}
S.~Kar and J.~M. Moura, ``Consensus+ innovations distributed inference over networks: cooperation and sensing in networked systems,'' {\em IEEE Signal Processing Magazine}, vol.~30, no.~3, pp.~99--109, 2013.

\bibitem{mateos2010distributed}
G.~Mateos, J.~A. Bazerque, and G.~B. Giannakis, ``Distributed sparse linear regression,'' {\em IEEE Transactions on Signal Processing}, vol.~58, no.~10, pp.~5262--5276, 2010.

\bibitem{lian2017can}
X.~Lian, C.~Zhang, H.~Zhang, C.-J. Hsieh, W.~Zhang, and J.~Liu, ``Can decentralized algorithms outperform centralized algorithms? a case study for decentralized parallel stochastic gradient descent,'' {\em Advances in neural information processing systems}, vol.~30, 2017.

\bibitem{nedic2009distributed}
A.~Nedic and A.~Ozdaglar, ``Distributed subgradient methods for multi-agent optimization,'' {\em IEEE Transactions on Automatic Control}, vol.~54, no.~1, pp.~48--61, 2009.

\bibitem{shi2015extra}
W.~Shi, Q.~Ling, G.~Wu, and W.~Yin, ``Extra: An exact first-order algorithm for decentralized consensus optimization,'' {\em SIAM Journal on Optimization}, vol.~25, no.~2, pp.~944--966, 2015.

\bibitem{qu2017harnessing}
G.~Qu and N.~Li, ``Harnessing smoothness to accelerate distributed optimization,'' {\em IEEE Transactions on Control of Network Systems}, vol.~5, no.~3, pp.~1245--1260, 2017.

\bibitem{zeng2018nonconvex}
J.~Zeng and W.~Yin, ``On nonconvex decentralized gradient descent,'' {\em IEEE Transactions on signal processing}, vol.~66, no.~11, pp.~2834--2848, 2018.

\bibitem{koloskova2020unified}
A.~Koloskova, N.~Loizou, S.~Boreiri, M.~Jaggi, and S.~Stich, ``A unified theory of decentralized sgd with changing topology and local updates,'' in {\em International conference on machine learning}, pp.~5381--5393, PMLR, 2020.

\bibitem{liu2024decentralized}
Y.~Liu, T.~Lin, A.~Koloskova, and S.~U. Stich, ``Decentralized gradient tracking with local steps,'' {\em Optimization Methods and Software}, pp.~1--28, 2024.

\bibitem{zhang2014asynchronous}
R.~Zhang and J.~Kwok, ``Asynchronous distributed admm for consensus optimization,'' in {\em International conference on machine learning}, pp.~1701--1709, PMLR, 2014.

\bibitem{chang2016asynchronous}
T.-H. Chang, M.~Hong, W.-C. Liao, and X.~Wang, ``Asynchronous distributed admm for large-scale optimization—part i: Algorithm and convergence analysis,'' {\em IEEE Transactions on Signal Processing}, vol.~64, no.~12, pp.~3118--3130, 2016.

\bibitem{li2019communication}
W.~Li, Y.~Liu, Z.~Tian, and Q.~Ling, ``Communication-censored linearized admm for decentralized consensus optimization,'' {\em IEEE Transactions on Signal and Information Processing over Networks}, vol.~6, pp.~18--34, 2019.

\bibitem{zhang2018admm}
C.~Zhang, M.~Ahmad, and Y.~Wang, ``Admm based privacy-preserving decentralized optimization,'' {\em IEEE Transactions on Information Forensics and Security}, vol.~14, no.~3, pp.~565--580, 2018.

\bibitem{shi2014linear}
W.~Shi, Q.~Ling, K.~Yuan, G.~Wu, and W.~Yin, ``On the linear convergence of the admm in decentralized consensus optimization,'' {\em IEEE Transactions on Signal Processing}, vol.~62, no.~7, pp.~1750--1761, 2014.

\bibitem{bastianello2020asynchronous}
N.~Bastianello, R.~Carli, L.~Schenato, and M.~Todescato, ``Asynchronous distributed optimization over lossy networks via relaxed admm: Stability and linear convergence,'' {\em IEEE Transactions on Automatic Control}, vol.~66, no.~6, pp.~2620--2635, 2020.

\bibitem{hajinezhad2019perturbed}
D.~Hajinezhad and M.~Hong, ``Perturbed proximal primal--dual algorithm for nonconvex nonsmooth optimization,'' {\em Mathematical Programming}, vol.~176, no.~1, pp.~207--245, 2019.

\bibitem{yi2021linear}
X.~Yi, S.~Zhang, T.~Yang, T.~Chai, and K.~H. Johansson, ``Linear convergence of first-and zeroth-order primal--dual algorithms for distributed nonconvex optimization,'' {\em IEEE Transactions on Automatic Control}, vol.~67, no.~8, pp.~4194--4201, 2021.

\bibitem{chang2020distributed}
T.-H. Chang, M.~Hong, H.-T. Wai, X.~Zhang, and S.~Lu, ``Distributed learning in the nonconvex world: From batch data to streaming and beyond,'' {\em IEEE Signal Processing Magazine}, vol.~37, no.~3, pp.~26--38, 2020.

\bibitem{nedic2014distributed}
A.~Nedi{\'c} and A.~Olshevsky, ``Distributed optimization over time-varying directed graphs,'' {\em IEEE Transactions on Automatic Control}, vol.~60, no.~3, pp.~601--615, 2014.

\bibitem{nedic2017achieving}
A.~Nedic, A.~Olshevsky, and W.~Shi, ``Achieving geometric convergence for distributed optimization over time-varying graphs,'' {\em SIAM Journal on Optimization}, vol.~27, no.~4, pp.~2597--2633, 2017.

\bibitem{wen2017terngrad}
W.~Wen, C.~Xu, F.~Yan, C.~Wu, Y.~Wang, Y.~Chen, and H.~Li, ``Terngrad: Ternary gradients to reduce communication in distributed deep learning,'' {\em Advances in neural information processing systems}, vol.~30, 2017.

\bibitem{de2015taming}
C.~M. De~Sa, C.~Zhang, K.~Olukotun, and C.~R{\'e}, ``Taming the wild: A unified analysis of hogwild-style algorithms,'' {\em Advances in neural information processing systems}, vol.~28, 2015.

\bibitem{dean2012large}
J.~Dean, G.~Corrado, R.~Monga, K.~Chen, M.~Devin, M.~Mao, M.~Ranzato, A.~Senior, P.~Tucker, K.~Yang, {\em et~al.}, ``Large scale distributed deep networks,'' {\em Advances in neural information processing systems}, vol.~25, 2012.

\bibitem{seide20141}
F.~Seide, H.~Fu, J.~Droppo, G.~Li, and D.~Yu, ``1-bit stochastic gradient descent and its application to data-parallel distributed training of speech dnns.,'' in {\em Interspeech}, vol.~2014, pp.~1058--1062, Singapore, 2014.

\bibitem{richtarik2021ef21}
P.~Richt{\'a}rik, I.~Sokolov, and I.~Fatkhullin, ``Ef21: A new, simpler, theoretically better, and practically faster error feedback,'' {\em Advances in Neural Information Processing Systems}, vol.~34, pp.~4384--4396, 2021.

\bibitem{reisizadeh2019exact}
A.~Reisizadeh, A.~Mokhtari, H.~Hassani, and R.~Pedarsani, ``An exact quantized decentralized gradient descent algorithm,'' {\em IEEE Transactions on Signal Processing}, vol.~67, no.~19, pp.~4934--4947, 2019.

\bibitem{magnusson2020maintaining}
S.~Magn{\'u}sson, H.~Shokri-Ghadikolaei, and N.~Li, ``On maintaining linear convergence of distributed learning and optimization under limited communication,'' {\em IEEE Transactions on Signal Processing}, vol.~68, pp.~6101--6116, 2020.

\bibitem{liu2021linear}
X.~Liu and Y.~Li, ``Linear convergent decentralized optimization with compression,'' in {\em International Conference on Learning Representations}, 2021.

\bibitem{zhao2022beer}
H.~Zhao, B.~Li, Z.~Li, P.~Richt{\'a}rik, and Y.~Chi, ``Beer: Fast $ o (1/t) $ rate for decentralized nonconvex optimization with communication compression,'' {\em Advances in Neural Information Processing Systems}, vol.~35, pp.~31653--31667, 2022.

\bibitem{koloskova2019decentralized}
A.~Koloskova, S.~Stich, and M.~Jaggi, ``Decentralized stochastic optimization and gossip algorithms with compressed communication,'' in {\em International Conference on Machine Learning}, pp.~3478--3487, PMLR, 2019.

\bibitem{koloskova2019decentralizedb}
A.~Koloskova, T.~Lin, S.~U. Stich, and M.~Jaggi, ``Decentralized deep learning with arbitrary communication compression,'' in {\em Proceedings of the 8th International Conference on Learning Representations}, 2019.

\bibitem{yau2022docom}
C.-Y. Yau and H.-T. Wai, ``Docom: Compressed decentralized optimization with near-optimal sample complexity,'' {\em arXiv preprint arXiv:2202.00255}, 2022.

\bibitem{xie2024communication}
A.~Xie, X.~Yi, X.~Wang, M.~Cao, and X.~Ren, ``A communication-efficient stochastic gradient descent algorithm for distributed nonconvex optimization,'' {\em arXiv preprint arXiv:2403.01322}, 2024.

\bibitem{yau2024fully}
C.-Y. Yau, H.~Liu, and H.-T. Wai, ``A stochastic approximation approach for efficient decentralized optimization on random networks,'' {\em arXiv preprint arXiv:2410.18774v2}, 2024.

\bibitem{srivastava2011distributed}
K.~Srivastava and A.~Nedic, ``Distributed asynchronous constrained stochastic optimization,'' {\em IEEE journal of selected topics in signal processing}, vol.~5, no.~4, pp.~772--790, 2011.

\bibitem{reisizadeh2023dimix}
H.~Reisizadeh, B.~Touri, and S.~Mohajer, ``Dimix: Diminishing mixing for sloppy agents,'' {\em SIAM Journal on Optimization}, vol.~33, no.~2, pp.~978--1005, 2023.

\bibitem{nassif2024differential}
R.~Nassif, S.~Vlaski, M.~Carpentiero, V.~Matta, and A.~H. Sayed, ``Differential error feedback for communication-efficient decentralized optimization,'' in {\em 2024 IEEE 13rd Sensor Array and Multichannel Signal Processing Workshop (SAM)}, pp.~1--5, IEEE, 2024.

\bibitem{michelusi2022finite}
N.~Michelusi, G.~Scutari, and C.-S. Lee, ``Finite-bit quantization for distributed algorithms with linear convergence,'' {\em IEEE Transactions on Information Theory}, vol.~68, no.~11, pp.~7254--7280, 2022.

\bibitem{nassif2023quantization}
R.~Nassif, S.~Vlaski, M.~Carpentiero, V.~Matta, M.~Antonini, and A.~H. Sayed, ``Quantization for decentralized learning under subspace constraints,'' {\em IEEE Transactions on Signal Processing}, vol.~71, pp.~2320--2335, 2023.

\bibitem{di2016next}
P.~Di~Lorenzo and G.~Scutari, ``Next: In-network nonconvex optimization,'' {\em IEEE Transactions on Signal and Information Processing over Networks}, vol.~2, no.~2, pp.~120--136, 2016.

\bibitem{mathkar2016nonlinear}
A.~S. Mathkar and V.~S. Borkar, ``Nonlinear gossip,'' {\em SIAM Journal on Control and Optimization}, vol.~54, no.~3, pp.~1535--1557, 2016.

\bibitem{liu2025two}
H.~Liu, C.-Y. Yau, and H.-T. Wai, ``A two-timescale primal-dual algorithm for decentralized optimization with compression,'' in {\em ICASSP}, 2025.

\bibitem{bertsekas1997nonlinear}
D.~P. Bertsekas, ``Nonlinear programming,'' {\em Journal of the Operational Research Society}, vol.~48, no.~3, pp.~334--334, 1997.

\bibitem{alistarh2017qsgd}
D.~Alistarh, D.~Grubic, J.~Li, R.~Tomioka, and M.~Vojnovic, ``Qsgd: Communication-efficient sgd via gradient quantization and encoding,'' {\em Advances in neural information processing systems}, vol.~30, 2017.

\bibitem{liao2024robust}
Y.~Liao, Z.~Li, S.~Pu, and T.-H. Chang, ``A robust compressed push-pull method for decentralized nonconvex optimization,'' {\em arXiv preprint arXiv:2408.01727}, 2024.

\bibitem{borkar1997stochastic}
V.~S. Borkar, ``Stochastic approximation with two time scales,'' {\em Systems \& Control Letters}, vol.~29, no.~5, pp.~291--294, 1997.

\bibitem{ghadimi2013stochastic}
S.~Ghadimi and G.~Lan, ``Stochastic first-and zeroth-order methods for nonconvex stochastic programming,'' {\em SIAM journal on optimization}, vol.~23, no.~4, pp.~2341--2368, 2013.

\bibitem{liu2020linear}
X.~Liu, Y.~Li, R.~Wang, J.~Tang, and M.~Yan, ``Linear convergent decentralized optimization with compression,'' {\em arXiv preprint arXiv:2007.00232}, 2020.

\bibitem{loshchilov2016sgdr}
I.~Loshchilov and F.~Hutter, ``Sgdr: Stochastic gradient descent with warm restarts,'' {\em arXiv preprint arXiv:1608.03983}, 2016.

\end{thebibliography}
\end{document}